\documentclass{article}

\usepackage{xparse}
\usepackage{amsmath}
\usepackage{amssymb}
\usepackage{enumitem}
\usepackage{mathrsfs}
\usepackage{fancyhdr}
\usepackage{pgf}

\title{Factorial-Type Recurrence Relations and $p$-adic Incomplete Gamma Functions}
\author{Paul Buckingham}
\date{}
\numberwithin{equation}{section}
\numberwithin{table}{section}
\newenvironment{proof}{\noindent\emph{Proof}.}{%
\hspace{\stretch{1}}\rule{1.5ex}{1.5ex} \vspace{5 mm}}
\newenvironment{remark}{\vspace{0.25cm} \noindent \textbf{Remark}.}{\vspace{0.25cm}}
\allowdisplaybreaks

\newtheorem{theorem}{Theorem}[section]
\newtheorem{prop}[theorem]{Proposition}
\newtheorem{lemma}[theorem]{Lemma}
\newtheorem{cor}[theorem]{Corollary}

\newtheorem{examp}[theorem]{Example}
\newcommand{\bex}{\begin{examp}\normalfont}
\newcommand{\eex}{\end{examp}}

\makeatletter
\newcommand\goinvisible{\pgfsys@begininvisible}
\newcommand\govisible{\pgfsys@endinvisible}
\makeatother

\newcommand{\spwr}{T}
\NewDocumentCommand{\lball}{mmo}{B_{< #2}(#1\IfValueTF{#3}{,#3}{})}
\newcommand{\qdo}{\Lambda}
\NewDocumentCommand{\bfact}{mmo}{%
  \mathscr{F}{%
    \IfValueTF{#3}{_{#3}}{}%
    (#1,#2)%
  }%
}
\newcommand{\sumf}{\Sigma}
\newcommand{\trvar}{s}
\newcommand{\inttmap}{\mathscr{I}}
\NewDocumentCommand{\inttrans}{sm}{%
  \inttmap
  \IfBooleanTF{#1}{%
    \mathopen{}\left(#2\right)\mathclose{}%
  }{%
    (#2)%
  }%
}
\newcommand{\gfn}[1]{\gamma_{#1}}
\newcommand{\decayegf}{\Cp\pwr{\gfvar}_0}
\newcommand{\lfn}{\phi}
\newcommand{\rfn}{\psi}

\newcommand{\lan}[2]{\mathrm{la}(#1,#2)}
\newcommand{\lconst}[2]{\mathrm{lc}(#1,#2)}

\newcommand{\qfsymb}{q}
\newcommand{\qf}{\qfsymb}
\newcommand{\qfh}{\widehat{\qfsymb}}

\newcommand{\ogen}[1]{G_{#1}} 
\newcommand{\egen}[1]{G_{*,#1}} 
\newcommand{\agen}[1]{H_{#1}} 
\NewDocumentCommand{\eqfp}{O{p}}{G_{*,\qf[#1]'}}
\NewDocumentCommand{\oqfh}{O{p}}{G_{\qfh[#1]}}
\NewDocumentCommand{\Oqfh}{O{p}}{\calG_{\qfh[#1]}}

\newcommand{\shiftu}{\sigma}
\newcommand{\gfvar}{t}

\newcommand{\ra}{\rightarrow}

\newcommand{\bb}[1]{\mathbb{#1}}

\newcommand{\Z}{\bb{Z}}
\newcommand{\Q}{\bb{Q}}
\newcommand{\R}{\bb{R}}
\newcommand{\C}{\bb{C}}
\newcommand{\Zp}{\Z_p}

\newcommand{\Qp}{\Q_p}

\newcommand{\F}{\mathbb{F}}
\NewDocumentCommand{\Fp}{O{p}}{\F_{#1}}

\NewDocumentCommand{\Fpbr}{O{p}}{\br{\F}_{#1}}

\newcommand{\of}{\circ}

\newcommand{\br}[1]{\bar{#1}} 

\providecommand{\ie}{i.e., }

\newcommand{\iso}{\cong}

\newcommand{\con}{\subseteq}

\newcommand{\eps}{\epsilon}

\newcommand{\sat}{\ |\ }

\newcommand{\Cp}{\mathbb{C}_p}

\newcommand{\im}{\operatorname{Im}}

\newcommand{\calL}{\mathcal{L}}

\newcommand{\pwr}[1]{[\![#1]\!]}

\newcommand{\calG}{\mathcal{G}}

\newcommand{\ch}[2]{{#1 \choose #2}}

\newcommand{\beq}{\begin{equation}}
\newcommand{\eeq}{\end{equation}}

\newcommand{\bea}{\begin{eqnarray*}}
\newcommand{\eea}{\end{eqnarray*}}
\newcommand{\beal}{\begin{eqnarray}}
\newcommand{\eeal}{\end{eqnarray}}
\newcommand{\bcs}{\left\{\begin{array}{ll}}
\newcommand{\ecs}{\end{array}\right.}

\newcommand{\one}{\mathbf{1}}

\newcommand{\calS}{\mathcal{S}}

\newcommand{\pair}[2]{\langle#1,#2\rangle}

\newcommand{\eqcom}[1]{\quad\text{#1}}
\NewDocumentCommand{\eqcomb}{mo}{\quad\text{(#1)\IfValueTF{#2}{#2}{}}}
\NewDocumentCommand{\mtr}{om}{%
  \IfValueTF{#1}{%
    \left(\begin{array}{@{}#1@{}}#2\end{array}\right)%
    }{%
    \begin{pmatrix}#2\end{pmatrix}%
  }%
}

\NewDocumentCommand{\dx}{sO{x}}{\IfBooleanTF{#1}{}{\,}d#2}
\NewDocumentCommand{\ddx}{O{x}o}{\frac{d\IfValueTF{#2}{^{#2}}{}}{d#1\IfValueTF{#2}{^{#2}}{}}}
\NewDocumentCommand{\tddx}{O{x}}{\tfrac{d}{d#1}}

\NewDocumentCommand{\Mat}{mmo}{%
  [#1]_{%
  \IfValueTF{#3}{%
    #3\leftarrow#2%
    }{%
    #2}%
  }%
}

\NewDocumentCommand{\pd}{omm}{\frac{\partial\IfValueTF{#1}{^#1}{}#2}{\partial#3\IfValueTF{#1}{^#1}{}}}

\newcommand{\eeie}{\text{\ie}\quad}

\newcommand{\cfns}[2]{C(#1,#2)}

\NewDocumentCommand{\rzrep}{sm}{\IfBooleanTF{#1}{\left\langle#2\right\rangle}{\langle#2\rangle}}

\NewDocumentCommand{\defmap}{mmmmmo}{%
  \bea
  #1 : #2 &\ra& #3 \\*\relax
  #4 &\mapsto& #5 \IfValueTF{#6}{#6}{}
  \eea
}
\NewDocumentCommand{\givemap}{ommmmo}{%
  \bea
  \IfValueTF{#1}{#1 : }{}#2 &\ra& #3 \\
  #4 &\mapsto& #5 \IfValueTF{#6}{#6}{}
  \eea
}

\NewDocumentCommand{\vint}{O{\Zp}}{\int_{#1}}

\newcommand{\cmod}[1]{\equiv_{\,#1}}

\newcommand{\emod}[1]{=_{\,\phantom{#1}}}

\NewDocumentCommand{\ballb}{mmo}{%
  \overline{B}%
  \IfValueTF{#3}{_{#3}}{}%
  (#1,#2)%
}
\NewDocumentCommand{\ballNb}{mmo}{%
  B%
  \IfValueTF{#3}{_{#3}}{}%
  (#1,#2)%
}

\def\BibTeX{\textrm{B\kern-.05em\textsc{i\kern-.025em b}\kern-.08em T\kern-.1667em\lower.5ex\hbox{E}\kern-.125emX}} 
\NewDocumentCommand{\stirlf}{smm}{\operatorname{\mathnormal{s}}\IfBooleanTF{#1}{'}{}(#2,#3)}

\NewDocumentCommand{\stirltf}{sm}{\operatorname{St}\IfBooleanTF{#1}{_1}{}(#2)}
\NewDocumentCommand{\stirltfinv}{sm}{\operatorname{St}^{-1}\IfBooleanTF{#1}{_1}{}(#2)}

\NewDocumentCommand{\lt}{sm}{%
  \calL
  \IfBooleanTF{#1}{%
    \left\{
  }{\{}%
  #2
  \IfBooleanTF{#1}{%
    \right\}
  }{\}}%
}

\begin{document}

\maketitle

\begin{abstract}
We introduce an automorphism $\calS$ of the space $\cfns{\Zp}{\Cp}$ of continuous functions $\Zp \ra \Cp$ and show that it can be used to give an alternative construction of the $p$-adic incomplete $\Gamma$-functions recently introduced by O'Desky and Richman~\cite{o'desky-richman:incomplete-gamma}. We then describe various properties of the automorphism $\calS$, showing that it is self-adjoint with respect to a certain non-degenerate symmetric bilinear form defined in terms of $p$-adic integration, and showing that its inverse plays a role in a $p$-adic integral-transform space akin to the role of differentiation in the classical space of Laplace-transformed functions. We also derive an integral-transform formula for the $p$-adic incomplete $\Gamma$-functions.

\medskip


\medskip

\end{abstract}

\section{Introduction}

In number theory, one is interested in finding $p$-adic analogues of classical functions of a real or complex variable. Important examples are $p$-adic $L$-functions and Morita's $p$-adic $\Gamma$-function. These $p$-adic analogues are usually defined via some interpolation property. More precisely, if $f$ is a complex function whose values at the positive integers can be viewed also as elements of the field $\Cp$, via a choice of isomorphism $\C \iso \Cp$, say, then one looks for a continuous $p$-adic function $F : \Zp \ra \Cp$ that takes the same values as $f$ on $\Z_{\geq 1}$. By the denseness of $\Z_{\geq 1}$ in $\Zp$, the continuous function $F$, if it exists, is uniquely determined. In practice, the function $f$ may need to be modified first, and the set of points at which one interpolates may be some other subset of $\Z$ dense in $\Zp$.

In a recent paper \cite{o'desky-richman:incomplete-gamma}, O'Desky and Richman construct a $p$-adic analogue of the incomplete $\Gamma$-function $\Gamma(-,r)$ for each $r \in 1 + p\Z$. Their construction uses the combinatorial machinery of derangements.

Here, we introduce an automorphism $\calS$ of the space $\cfns{\Zp}{\Cp}$ of continuous functions, restricting to an automorphism of the space $\lan{\Zp}{\Cp}$ of locally analytic functions $\Zp \ra \Cp$, and show that O'Desky and Richman's $p$-adic incomplete $\Gamma$-function $\Gamma_p(-,r)$ can be constructed by applying $\calS^{-1}$ to the locally analytic $p$-adic function $x \mapsto r^x$ ($r \in 1 + p\Z$). We thus give a construction of O'Desky and Richman's functions that does not use derangements.

We then turn to proving properties of the automorphism $\calS$. We show in Section~\ref{sec: self-adj} that it is self-adjoint with respect to a non-degenerate symmetric bilinear form on $\cfns{\Zp}{\Cp}$ defined in terms of $p$-adic integration. We also show that, with respect to a convolution product $\star$ on $\cfns{\Zp}{\Cp}$, $\calS$ again enjoys a sort of self-adjointness; see Section~\ref{sec: conv of mahl}.

Next, in Section~\ref{sec: integral transform}, we introduce an integral transform $\inttmap : \cfns{\Zp}{\Cp} \ra \lan{\Zp}{\Cp}$, again related to $p$-adic integration, describe its relationship with $\calS$, derive an integral-transform formula for the $p$-adic incomplete $\Gamma$-functions, and give a complete description of the image of $\inttmap$. We also prove a convolution property of $\inttmap$.

Finally, in Section~\ref{DE section}, we relate $\calS$ to solutions of the well-known differential equation $F' + F = G$. For example, if $G \in \Z\pwr{\gfvar}$, then the unique solution $F \in \Zp\pwr{\gfvar}$ can be constructed concretely via $\calS^{-1}$. We work in a more general setting than $\Zp$. The precise result is Proposition~\ref{DE prop}.

Running throughout is the function $\qf = \calS^{-1}(\one)$, where $\one : x \mapsto 1$. Up to a certain constant multiplicative factor and a one-unit shift in the argument, $\qf$ is the $p$-adic analogue of $\Gamma(-,1)$. It features prominently in results on the convolution product $\star$ and in the description of the image of $\inttmap$. The function $\qf$ has been well studied in various contexts (see Section~\ref{def qf} for a sample of works), and we hope our contribution will provide further evidence that it may have a role to play in $p$-adic analysis.

\subsection{Notation and definitions}

Throughout, $|\cdot|$ and $v_p$ will denote, respectively, the absolute value and the valuation on $\Cp$, normalized such that $|p| = 1/p$ and $v_p(p) = 1$. We will occasionally make reference to the balls $\lball{a}{\rho}[K] = \{x \in K \sat |x - a| < \rho\}$, where $K$ is a subfield of $\Cp$, $a \in K$, and $\rho$ is a positive real number.

Frequently used will be the functions $\beta_n : \Zp \ra \Zp$ ($n \geq 0$) defined by $\beta_n(x) = \ch{x}{n} = \frac{1}{n!}(x)_n$, where $(x)_n = x(x - 1)(x - 2) \cdots (x - n + 1)$. They satisfy $\nabla\beta_n = \beta_{n - 1}$ for all $n \geq 1$, where $\nabla$ is the forward-difference operator, \ie $(\nabla \lfn)(x) = \lfn(x + 1) - \lfn(x)$ for a function $\lfn$ on $\Zp$. By Mahler's Theorem, the continuous functions $\lfn : \Zp \ra \Cp$ are the functions that may be expressed as $\lfn = \sum_{n = 0}^\infty b_n\beta_n$ where the elements $b_n \in \Cp$ converge to zero.

If $\lfn : \Zp \ra \Cp$, we will be interested in the functions
\defmap{\widehat{\lfn}}{\Zp}{\Cp}{x}{\lfn(-1 - x)}
and
\defmap{\shiftu(\lfn)}{\Zp}{\Cp}{x}{\lfn(x + 1)}[.]
We will also refer often to the constant function $\one : \Zp \ra \Cp$, $x \mapsto 1$.

\section{The automorphism $\calS$}

Let $\cfns{\Zp}{\Cp}$ denote the $\Cp$-vector space of continuous functions $\Zp \ra \Cp$. We define a map
\[ \calS : \cfns{\Zp}{\Cp} \ra \cfns{\Zp}{\Cp} \]
as follows. If $\lfn \in \cfns{\Zp}{\Cp}$, then $\calS(\lfn)$ is the function $\Zp \ra \Cp$ defined by
\[ \calS(\lfn)(x) = \lfn(x) - x\lfn(x - 1) .\]
It is clear that $\calS(\lfn)$ is again continuous.

\begin{lemma} \label{calS of series}
Suppose that $\lfn_0,\lfn_1,\lfn_2,\ldots$ is a sequence of functions in $\cfns{\Zp}{\Cp}$ that converges uniformly to zero. Then the function $\lfn = \sum_{k = 0}^\infty \lfn_k$ is in $\cfns{\Zp}{\Cp}$, and $\calS(\lfn) = \sum_{k = 0}^\infty \calS(\lfn_k)$.
\end{lemma}

\begin{proof}
The continuity of $\lfn$ follows from the assumption of uniform convergence. For the second assertion, observe that for $x \in \Zp$,
\begin{align*}
\calS(\lfn)(x) &= \lfn(x) - x\lfn(x - 1) \\
&= \sum_{k = 0}^\infty \lfn_k(x) - x\sum_{k = 0}^\infty \lfn_k(x - 1) \\
&= \sum_{k = 0}^\infty (\lfn_k(x) - x\lfn_k(x - 1)) \\
&= \sum_{k = 0}^\infty \calS(\lfn_k)(x) .
\end{align*}
\end{proof}

\begin{prop} \label{calS of mahler}
If $\lfn = \sum_{n = 0}^\infty b_n\beta_n$ is continuous, then
\[ \calS(\lfn) = \sum_{n = 0}^\infty (b_n - nb_{n - 1})\beta_n ,\]
where we have set $b_{-1} = 0$.
\end{prop}

\begin{proof}
First note that
\[ \calS(\beta_n)(x) = \ch{x}{n} - x\ch{x - 1}{n} = \ch{x}{n} - (n + 1)\ch{x}{n + 1} ,\]
so $\calS(\beta_n) = \beta_n - (n + 1)\beta_{n + 1}$. Hence,
\begin{align*}
\calS(\lfn) &= \sum_{n = 0}^\infty b_n(\beta_n - (n + 1)\beta_{n + 1}) \eqcom{by Lemma~\ref{calS of series}} \\
&= \sum_{n = 0}^\infty b_n\beta_n - \sum_{n = 1}^\infty nb_{n - 1}\beta_n \\
&= \sum_{n = 0}^\infty (b_n - nb_{n - 1})\beta_n .
\end{align*}
\end{proof}

\begin{lemma} \label{preimage of beta_k}
If $k \geq 0$, then
\[ \calS\left(\sum_{n = k}^\infty \frac{n!}{k!}\beta_n\right) = \beta_k .\]
\end{lemma}

\begin{proof}
Let $\lfn = \sum_{n = k}^\infty \tfrac{n!}{k!}\beta_n$. Note that $\lfn$ is a well-defined function in $\cfns{\Zp}{\Cp}$, because its sequence of Mahler coefficients converges to $0$. Then by Proposition~\ref{calS of mahler},
\[ \calS(\lfn) = \beta_k + \sum_{n = k + 1}^\infty \left(\frac{n!}{k!} - n\,\frac{(n - 1)!}{k!}\right)\beta_n = \beta_k + \sum_{n = k + 1}^\infty \left(\frac{n!}{k!} - \frac{n!}{k!}\right)\beta_n = \beta_k .\]
\end{proof}

\begin{theorem} \label{calS theorem}
The map $\calS$ is an automorphism of the $\Cp$-vector space $\cfns{\Zp}{\Cp}$. If $\rfn = \sum_{k = 0}^\infty b_k\beta_k$, then
\beq \label{formula for preimage}
\calS^{-1}(\rfn) = \sum_{n = 0}^\infty \left(\sum_{k = 0}^n \frac{n!}{k!}b_k\right)\beta_n .
\eeq
\end{theorem}

\begin{proof}
For injectivity, suppose that $\calS(\lfn) = 0$, \ie $\lfn(x) = x\lfn(x - 1)$ for all $x \in \Zp$. We show by induction on $n$ that $\lfn(n) = 0$ for all $n \in \Z_{\geq 0}$, from which it follows that $\lfn = 0$ by continuity. The base case is clear, because $\lfn(0) = 0 \cdot \lfn(-1) = 0$. Now let $n \geq 0$ and assume that $\lfn(n) = 0$. Then $\lfn(n + 1) = (n + 1)\lfn(n) = 0$. This completes the induction and therefore the proof of injectivity.

For surjectivity, let $\rfn : \Zp \ra \Cp$ be any continuous function, and let $\rfn = \sum_{k = 0}^\infty b_k \beta_k$ be its Mahler series, so that $b_k \ra 0$. We already know from Lemma~\ref{preimage of beta_k} that $\calS(\lfn_k) = \beta_k$, where $\lfn_k = \sum_{n = k}^\infty \tfrac{n!}{k!}\beta_n$. Because the functions $\lfn_k$ are all maps into $\Zp$, they have sup-norm $\|\lfn_k\| \leq 1$, so the sequence of functions $b_k\lfn_k$ converges uniformly to zero. Hence, by uniform convergence, we have a well-defined, continuous function
\[ \lfn = \sum_{k = 0}^\infty b_k\lfn_k \]
from $\Zp$ to $\Cp$. Then
\begin{align*}
\calS(\lfn) &= \sum_{k = 0}^\infty \calS(b_k\lfn_k) \eqcom{by Lemma~\ref{calS of series}} \\
&= \sum_{k = 0}^\infty b_k \beta_k \\
&= \rfn .
\end{align*}

As for the formula in (\ref{formula for preimage}), if $x \in \Zp$, then
\[ \lfn(x) = \sum_{k = 0}^\infty b_k\sum_{n = k}^\infty \frac{n!}{k!}\beta_n(x) = \sum_{n = 0}^\infty \left(\sum_{k = 0}^n \frac{n!}{k!}b_k\right)\beta_n(x) ,\]
the interchanging of the summation signs being justified as follows. Fix $x \in \Zp$ and let
\[ \eta(n,k) =
\begin{cases}
\tfrac{n!}{k!}b_k\beta_n(x) & \text{if $k \leq n$,} \\
0 & \text{otherwise.}
\end{cases}
\]
Note that $|\eta(n,k)| \leq |\tfrac{n!}{k!}|\,|b_k|$. By \cite[Chap.~2, Sect.~1.2, Cor.]{robert:p-adic}, it is enough to show that, for any $\eps > 0$, there are only finitely many pairs $(n,k)$ such that $|\eta(n,k)| > \eps$. Let $\eps > 0$, then. Because $b_k \to 0$, there is $M \geq 0$ such that $|b_k| \leq \eps$ for all $k \geq M$. For each $k < M$, choose $N_k \geq 0$ such that $|\tfrac{n!}{k!}|\,|b_k| \leq \eps$ for all $n \geq N_k$, possible because $|n!/k!| \to 0$ as $n \to \infty$. Now let $N = \max(N_0,\ldots,N_{M - 1})$. If $k \geq M$, then $|\eta(n,k)| \leq |b_k| \leq \eps$. On the other hand, if $k < M$ and $n \geq N$, then either $k > n$, in which case $\eta(n,k) = 0$, or $k \leq n$, and then $|\eta(n,k)| \leq |\tfrac{n!}{k!}|\,|b_k| \leq \eps$ because $n \geq N \geq N_k$. Thus, $|\eta(n,k)| \leq \eps$ except possibly for the finitely many pairs $(n,k) \in \{0,\ldots,N - 1\} \times \{0,\ldots,M - 1\}$.
\end{proof}

\subsection{Locally analytic functions}

We use the definition of \emph{locally analytic} given in \cite[Sect.~I.4]{colmez:fonctions}. The following proposition is \cite[Cor.~I.4.8]{colmez:fonctions}. See also \cite[Sect.~10, Cor.~1]{amice:interpolation} in the case where $\Zp$ maps into $\Qp$.

\begin{prop} \label{amice's criterion}
If $f : \Zp \ra \Cp$ has Mahler coefficients $(a_n)_n$, then $f$ is locally analytic on $\Zp$ if and only if $\liminf\limits_{n \to \infty} \frac{v_p(a_n)}{n} > 0$. Equivalently, if we set $a_n' = a_n/n!$, then $f$ is locally analytic on $\Zp$ if and only if $\liminf\limits_{n \to \infty} \frac{v_p(a_n')}{n} > -\frac{1}{p - 1}$.
\end{prop}

The following equivalent form of Proposition~\ref{amice's criterion} is an immediate consequence of the fact that, if $(x_n)_n$ is a sequence of real numbers, then the sequence $(\inf_{k \geq n} x_k)_n$ is monotone non-decreasing.

\begin{prop} \label{amice variant}
If $f : \Zp \ra \Cp$ has Mahler coefficients $(a_n)_n$, then $f$ is locally analytic on $\Zp$ if and only if there are $\beta > 0$ and $N \in \Z_{\geq 0}$ such that $v_p(a_n) \geq n\beta$ for all $n \geq N$. Equivalently, if we set $a_n' = a_n/n!$, then $f$ is locally analytic on $\Zp$ if and only if there are $\beta > -\frac{1}{p - 1}$ and $N \in \Z_{\geq 0}$ such that $v_p(a_n') \geq n\beta$ for all $n \geq N$.
\end{prop}

We use Proposition~\ref{amice variant} to prove the following.

\begin{theorem} \label{calS on la}
The map $\calS$ restricts to an automorphism of the $\Cp$-vector space $\lan{\Zp}{\Cp}$.
\end{theorem}

\begin{proof}
It is clear from the definition of $\calS$ that $\calS(\lfn) \in \lan{\Zp}{\Cp}$ if $\lfn$ is. It is also clear from the injectivity of $\calS$ as a map on $\cfns{\Zp}{\Cp}$ that it is still injective when restricted to $\lan{\Zp}{\Cp}$. It remains to show that if $\rfn : \Zp \ra \Cp$ is locally analytic, then the unique $\lfn \in \cfns{\Zp}{\Cp}$ such that $\calS(\lfn) = \rfn$ is also locally analytic.

Let the Mahler coefficients of $\rfn$ be $(b_n)_{n \geq 0}$. By Theorem~\ref{calS theorem}, the $n$th Mahler coefficient of $\lfn$ is
\[ a_n = \sum_{k = 0}^n \frac{n!}{k!}b_k .\]
Letting $a_n' = a_n/n!$ and $b_n' = b_n/n!$, we therefore have
\beq \label{a_n' and b_k'}
a_n' = \sum_{k = 0}^n b_k' .
\eeq
The assumption that $\rfn$ is locally analytic says, by Proposition~\ref{amice variant}, that there are $\beta > -\frac{1}{p - 1}$ and $N \in \Z_{\geq 0}$ such that, for all $n \geq N$,
\[ v_p(b_n') \geq n\beta .\]

Assume now that $n \geq N$, and write
\beq \label{a_n' split up}
a_n' = s + \sum_{k = N}^n b_k' ,
\eeq
where $s = \sum_{k = 0}^{N - 1} b_k'$. Let us first consider the case where $\beta \geq 0$, in which case $v_p(b_k') \geq 0$ for all $k \geq N$. In this case, $v_p(a_n') \geq \min(v_p(s),0)$, so $v_p(a_n')/n \geq \frac{1}{n}\min(v_p(s),0)$, which tends to $0$ as $n \to \infty$. Therefore, if $\alpha$ is any negative real number, $v_p(a_n') \geq n\alpha$ for $n$ sufficiently large. Thus, $\lfn$ is locally analytic by Proposition~\ref{amice variant}, because we may choose $\alpha$ such that $-1/(p - 1) < \alpha < 0$.

We may thus assume for the remainder of the proof that $\beta < 0$. Then because $v_p(s)/n \to 0$, there is $K \geq 1$ such that, if $n \geq K$, then $v_p(s)/n \geq \beta$, \ie
\[ v_p(s) \geq n\beta .\]
Further, each $b_k'$ in the sum in (\ref{a_n' split up}) has valuation greater than or equal to $k\beta$ because $k \geq N$ in the sum. But $k\beta \geq n\beta$, because $\beta < 0$ and $k \leq n$ for $k$ in the sum. Thus, the summation term in (\ref{a_n' split up}) has valuation at least $n\beta$ as well. Therefore, $v_p(a_n') \geq n\beta$ for all $n \geq \max(N,K)$, and so $\lfn$ is locally analytic by Proposition~\ref{amice variant}.
\end{proof}

\subsection{An alternative description of $\calS^{-1}(\rfn)$}

If $\rfn : \Zp \ra \Cp$ is continuous, then because $\Zp$ is compact, the same is true of $\rfn(\Zp)$, which is therefore bounded, say $|\rfn(x)| \leq M$ for all $x \in \Zp$. Consequently, for each $n \geq 0$,
\[ \sup_{x \in \Zp} |(x)_n\rfn(x - n)| = \sup_{x \in \Zp} \left|n!\,\ch{x}{n}\rfn(x - n)\right| \leq |n!|M ,\]
and the real numbers $|n!|M$ converge to $0$ as $n \to \infty$. Therefore, by uniform convergence, we have a well-defined, continuous function $\Zp \ra \Cp$ given by
\[ x \mapsto \sum_{n = 0}^\infty (x)_n\rfn(x - n) .\]

\begin{prop} \label{alternative for calS^{-1}}
If $\rfn \in \cfns{\Zp}{\Cp}$, then
\[ \calS^{-1}(\rfn)(x) = \sum_{n = 0}^\infty (x)_n\rfn(x - n) \]
for all $x \in \Zp$.
\end{prop}

\begin{proof}
Let $\lfn(x) = \sum_{n = 0}^\infty (x)_n\rfn(x - n)$. Then
\begin{align*}
\calS(\lfn)(x) &= \sum_{n = 0}^\infty (x)_n\rfn(x - n) - x\sum_{n = 0}^\infty (x - 1)_n\rfn(x - 1 - n) \\
&= \sum_{n = 0}^\infty (x)_n\rfn(x - n) - \sum_{n = 0}^\infty (x)_{n + 1}\rfn(x - (n + 1)) \\
&= \sum_{n = 0}^\infty (x)_n\rfn(x - n) - \sum_{n = 1}^\infty (x)_n\rfn(x - n) \\
&= \rfn(x) .
\end{align*}
Thus, $\calS(\lfn) = \rfn$, so $\calS^{-1}(\rfn) = \lfn$.
\end{proof}

\subsection{Relationship to factorial-type recurrence relations}

Consider a first-order linear recurrence relation
\beq \label{first-order lin}
a_n = h(n)a_{n - 1} + \rfn(n) \eqcomb{$n \geq 1$}
\eeq
where $h,\rfn : \Zp \ra \Cp$ are continuous functions and the $a_n$ are in $\Cp$. Let us say that a solution $(a_n)_{n \geq 0}$ to such a recurrence relation is \emph{continuous} if there is a continuous function $\lfn : \Zp \ra \Cp$ such that $\lfn(n) = a_n$ for all $n \geq 0$.

By an \emph{automatically seeded} recurrence relation, we will mean a recurrence relation as in (\ref{first-order lin}) for which the function $h$ has a zero $n_0 \in \Z$. Note that $n_0$ may be positive, negative, or zero. The terminology is explained via the following proposition and its proof.

\begin{prop} \label{unique cont sol to auto}
If an automatically seeded recurrence relation admits a continuous solution, then the solution is unique.
\end{prop}

\begin{proof}
Suppose that $(a_n)_{n \geq 0}$ is a continuous solution to $a_n = h(n)a_{n - 1} + \rfn(n)$ ($n \geq 1$), so that there is a continuous function $\lfn : \Zp \ra \Cp$ such that $\lfn(n) = a_n$ for all $n \geq 0$. Then for $n \geq 1$,
\[ \lfn(n) = a_n = h(n)a_{n - 1} + \rfn(n) = h(n)\lfn(n - 1) + \rfn(n) ,\]
so because $\lfn$, $\rfn$, and $h$ are all continuous, we in fact have $\lfn(x) = h(x)\lfn(x - 1) + \rfn(x)$ for all $x \in \Zp$. In particular,
\[ \lfn(n_0) = h(n_0)\lfn(n_0 - 1) + \rfn(n_0) = \rfn(n_0) \]
by the assumption that $h(n_0) = 0$. Hence, if $(\tilde{a}_n)_{n \geq 0}$ is another continuous solution, and if $\lfn_1(n) = \tilde{a}_n$ for all $n \geq 0$, then $\lfn(n_0) = \lfn_1(n_0)$, so $\lfn(n) = \lfn_1(n)$ for all $n \geq n_0$, and so $\lfn = \lfn_1$ by continuity. Thus, $a_n = \tilde{a}_n$ for all $n \geq 0$.
\end{proof}

In particular, the value $a_{n_0}$ in an automatically seeded recurrence relation admitting a continuous solution is determined by the recurrence relation itself, unlike in an ordinary recurrence relation, where $a_{n_0}$ may typically be chosen freely.

For us, the function $h$ of interest will be $h(x) = x$, so that the recurrence relation becomes $a_n = na_{n - 1} + \rfn(n)$. It seems reasonable to call this a \emph{factorial-type} recurrence relation, because one recovers the usual factorial recurrence relation by taking $\rfn = 0$.

Theorem~\ref{calS theorem} says that every factorial-type recurrence relation $a_n = na_{n - 1} + \rfn(n)$, where $\rfn \in \cfns{\Zp}{\Cp}$, has the unique continuous solution $a_n = \calS^{-1}(\rfn)(n)$. Thus, every time we invoke the inverse map $\calS^{-1}$ in this paper, we are essentially considering the unique continuous solution to some factorial-type recurrence relation. Although this point of view will not be emphasized, the reader may well like to keep it in mind while reading the rest of the paper.

\subsection{The function $\qf$} \label{def qf}

Central to this paper is the function $\qf = \calS^{-1}(\one)$, where $\one : \Zp \ra \Cp$ is the constant function $x \mapsto 1$. By Proposition~\ref{alternative for calS^{-1}}, $\qf(x) = \sum_{n = 0}^\infty (x)_n = \sum_{n = 0}^\infty n!\ch{x}{n}$. By definition, it satisfies $\qf(x) = x\qf(x - 1) + 1$ for all $x \in \Zp$ and therefore gives the unique continuous solution to the recurrence relation $a_n = na_{n - 1} + 1$. We note as well that it is locally analytic by Theorem~\ref{calS on la}. (See also \cite{pasol-zaharescu:sondow}, where the function is called $A$.)

This function has been studied many times already, in several contexts and with various symbols to denote it. Aside from the reference \cite{pasol-zaharescu:sondow} just mentioned, we point out just a few more: \cite[Lemma~5]{halbeisen-hungerbuhler:comb-func}, where a connection to the incomplete $\Gamma$-function $\Gamma(-,1)$ is made, and \cite{sondow-schalm:partial} and \cite{berndt-kim-zaharescu:dioph}, which consider convergents to $e$, the latter proving a conjecture of Sondow.

\subsection{$p$-adic incomplete $\Gamma$-functions} \label{sec: p-adic inc}

In \cite{o'desky-richman:incomplete-gamma}, O'Desky and Richman define a $p$-adic incomplete $\Gamma$-function $\Gamma_p(-,r)$ for each $r \in 1 + p\Zp$ via the combinatorial machinery of derangements. Their function satisfies a $p$-adic interpolation property with respect to the usual (complex) incomplete $\Gamma$-function $\Gamma(-,r)$. We give a construction using instead the invertible map $\calS : \cfns{\Zp}{\Cp} \ra \cfns{\Zp}{\Cp}$.

For each $r \in \lball{1}{1}[\Cp]$, define
\defmap{g_r}{\Zp}{\Cp}{x}{r^x}[,]
locally analytic by Proposition~\ref{amice's criterion}. Let $\gfn{r} = \calS^{-1}(g_r)$, also locally analytic by Theorem~\ref{calS on la}. Define $\tilde{p}$ to be $p$ if $p$ is odd and to be $4$ if $p = 2$, and define
\defmap{\overline{\gamma}_r}{\Zp}{\Cp}{x}{\exp(r\tilde{p})\gfn{r}(x - 1)}[,]
again obviously locally analytic, because $\gfn{r}$ is. The function $\overline{\gamma}_r$ will give us the desired interpolation property.

By definition of $\calS$ and $g_r$,
\begin{align*}
\gfn{r}(x) &= x\gfn{r}(x - 1) + r^x \eqcom{for all $x \in \Zp$,} \\
\eeie \exp(r\tilde{p})\gfn{r}(x) &= x\exp(r\tilde{p})\gfn{r}(x - 1) + r^x\exp(r\tilde{p}) \eqcom{for all $x \in \Zp$,} \\
\eeie \overline{\gamma}_r(x + 1) &= x\overline{\gamma}_r(x) + r^x\exp(r\tilde{p}) \eqcom{for all $x \in \Zp$.}
\end{align*}
In particular,
\beq \label{my p-adic inc gamma rec rel}
\overline{\gamma}_r(n + 1) = n\overline{\gamma}_r(n) + r^n\exp(r\tilde{p}) \eqcom{for all $n \in \Z_{\geq 0}$.}
\eeq
Note also that $\overline{\gamma}_r(1) = 0 \cdot \overline{\gamma}_r(0) + \exp(r\tilde{p}) = \exp(r\tilde{p})$.

Now suppose that $r \in \lball{1}{1}[\Cp] \cap \Z = 1 + p\Z$. It is well known that the ordinary (complex) incomplete $\Gamma$-function $\Gamma(-,r)$ maps $\Z_{\geq 1}$ into $\Q(e)$ and satisfies
\beq \label{inc gamma rec rel}
\Gamma(n + 1,r) = n\Gamma(n,r) + r^n e^{-r} \eqcom{for all $n \in \Z_{\geq 1}$.}
\eeq
Consider, then, the unique field map $\tau_p : \Q(e) \ra \Qp$ mapping $e^{-1}$ to $\exp(\tilde{p})$. (The map $\tau_p$ appears already in O'Desky and Richman's paper; we do not claim $\tau_p$ as our own.) Applying $\tau_p$ to both sides of (\ref{inc gamma rec rel}), we have
\[ \tau_p(\Gamma(n + 1,r)) = n\tau_p(\Gamma(n,r)) + r^n\exp(r\tilde{p}) \eqcom{for all $n \in \Z_{\geq 1}$.} \]
Thus, the values $\tau_p(\Gamma(n,r))$, $n \in \Z_{\geq 1}$, satisfy the same recurrence relation as the values $\overline{\gamma}_r(n)$, as in (\ref{my p-adic inc gamma rec rel}). Further,
\[ \overline{\gamma}_r(1) = \exp(r\tilde{p}) = \tau_p(e^{-r}) = \tau_p(\Gamma(1,r)) ,\]
so in fact $\overline{\gamma}_r(n) = \tau_p(\Gamma(n,r))$ for all $n \geq 1$. Thus, the $p$-adic locally analytic function $\overline{\gamma}_r$ interpolates the values $\tau_p(\Gamma(n,r))$ at positive integers $n$. The existence of a $p$-adic locally analytic function with this interpolation property is what O'Desky and Richman showed, but via a different method.

Of course, since $g_1 = \one$, we have $\gfn{1} = \qf$.

\begin{remark}
We point out a small discrepancy between our function $\overline{\gamma}_r$ and O'Desky and Richman's $\Gamma_p(-,r)$ in the case where $p = 2$. They define the map $\tau_p : \Q(e) \ra \Qp$ by sending $e^{-1}$ to $\exp(p)$, while our choice is to map $e^{-1}$ to $\exp(\tilde{p})$. The maps are the same when $p$ is odd, but our choice is made to ensure that the exponential series $\exp(\tilde{p})$ converges even in the case $p = 2$, where $\tilde{p} = 4$. (Note that $\exp(x)$ converges if and only if $|x| < |p|^{1/(p - 1)}$.) The exceptional nature of the prime $2$ often requires the introduction of an extra factor of $2$, especially in Iwasawa theory. See \cite[Sect.~7.2]{wash:cyc}, for example. Another option might be to employ a continuous extension of $\exp$ to $\C_2$, so that $\exp(2)$ is defined for that extension. However, there is no canonical such extension. See \cite[Chap.~5, Sect.~4.4]{robert:p-adic} for a discussion.
\end{remark}

We will provide an integral-transform formula for $\overline{\gamma}_r$ in Section~\ref{int transf for gamma}.

\subsubsection{Mahler series of $\overline{\gamma}_r(x + 1)$}

\begin{prop}
Let $r \in \lball{1}{1}[\Cp]$. Then the Mahler series of $\overline{\gamma}_r(x + 1)$ is
\[ \overline{\gamma}_r(x + 1) = \sum_{n = 0}^\infty \left(\exp(r\tilde{p})\sum_{k = 0}^n \frac{n!}{k!}(r - 1)^k\right)\ch{x}{n} .\]
\end{prop}

\begin{proof}
Because $\gfn{r} = \calS^{-1}(g_r)$ and $g_r = \sum_{n = 0}^\infty (r - 1)^n\beta_n$, Theorem~\ref{calS theorem} shows that
\[ \gfn{r}(x) = \sum_{n = 0}^\infty \left(\sum_{k = 0}^n \frac{n!}{k!}(r - 1)^k\right)\ch{x}{n} .\]
Hence,
\[ \overline{\gamma}_r(x + 1) = \exp(r\tilde{p})\gfn{r}(x) = \sum_{n = 0}^\infty \left(\exp(r\tilde{p})\sum_{k = 0}^n \frac{n!}{k!}(r - 1)^k\right)\ch{x}{n} .\]
\end{proof}

\subsubsection{Recovering O'Desky and Richman's formula}

Since $\overline{\gamma}_r$ and $\Gamma_p(-,r)$ interpolate the same $p$-adic values at the positive integers (with the above remark about the $p = 2$ case understood), they must be the same function, and indeed, the following formula for $\overline{\gamma}$ concords with that given for $\Gamma_p(-,r)$ in \cite[Theorem~1.1]{o'desky-richman:incomplete-gamma}.

\begin{prop}
Let $r \in \lball{1}{1}[\Cp]$. Then
\[ \overline{\gamma}_r(x) = \exp(r\tilde{p})\sum_{n = 0}^\infty (x - 1)_n r^{x - 1 - n} = \exp(r\tilde{p})\sum_{n = 0}^\infty n!\,\ch{x - 1}{n}r^{x - 1 - n} .\]
\end{prop}

\begin{proof}
Because $\gfn{r} = \calS^{-1}(g_r)$, Proposition~\ref{alternative for calS^{-1}} shows that
\[ \gfn{r}(x) = \sum_{n = 0}^\infty (x)_n r^{x - n} .\]
Therefore, since $\overline{\gamma}_r(x) = \exp(r\tilde{p})\gfn{r}(x - 1)$, the formula follows.
\end{proof}

\subsection{$1$-Lipschitz functions}

Recall that a function $\lfn : \Zp \ra \Cp$ is called $1$-Lipschitz if $|\lfn(x) - \lfn(y)| \leq |x - y|$ for all $x,y \in \Zp$. Such a function is obviously continuous.

\begin{prop} \label{lip prop}
Let $A = \{x \in \Cp \sat |x| \leq 1\}$. The map $\calS$ restricts to an automorphism of the $A$-module of $1$-Lipschitz functions $\Zp \ra A$.
\end{prop}

\begin{proof}
First, suppose that $\lfn : \Zp \ra A$ is $1$-Lipschitz. It is clear from the definition of $\calS(\lfn)$ that its image is again in $A$. For the $1$-Lipschitz property of $\calS(\lfn)$, we take $x,y \in \Zp$. Then
\begin{align*}
\goinvisible &=\govisible |\calS(\lfn)(x) - \calS(\lfn)(y)| \\
&= |\lfn(x) - x\lfn(x - 1) - \lfn(y) + y\lfn(y - 1)| \\
&= |(\lfn(x) - \lfn(y)) - x(\lfn(x - 1) - \lfn(y - 1)) + (y - x)\lfn(y - 1)| \\
&\leq \max(|\lfn(x) - \lfn(y)|,|x|\,|\lfn(x - 1) - \lfn(y - 1)|,|y - x|\,|\lfn(y - 1)|) \\
&\leq \max(|x - y|,|x|\,|x - y|,|x - y|\,|\lfn(y - 1)|) \\
&\leq |x - y|
\end{align*}
by the assumption that the image of $\lfn$ is contained in $A$.

Conversely, suppose that $\rfn : \Zp \ra A$ is $1$-Lipschitz, and let $\lfn = \calS^{-1}(\rfn)$, so that $\lfn(x) = x\lfn(x - 1) + \rfn(x)$ for all $x \in \Zp$. It is clear from Proposition~\ref{alternative for calS^{-1}} that $\lfn$ has its image in $A$. As for the $1$-Lipschitz property of $\lfn$, it is equivalent to the statement that for all $a,x \in \Zp$, $|\lfn(x + a) - \lfn(x)| \leq |a|$. To that end, let us fix $a \in \Zp$ and introduce an equivalence relation $\equiv_a$ on $A$ defined by $\alpha \equiv_a \beta$ if $|\alpha - \beta| \leq |a|$. (Transitivity holds because of the ultrametric property.)

We first show by induction on $x$ that if $x \in \Z_{\geq 0}$, then $\lfn(x + a) \equiv_a \lfn(x)$. For the case $x = 0$, observe that
\begin{align*}
\lfn(a) &\emod{a} a\lfn(a - 1) + \rfn(a) \\
&\cmod{a} \rfn(a) \eqcom{because $|\lfn(a - 1)| \leq 1$} \\
&\cmod{a} \rfn(0) \eqcom{because $\rfn$ is $1$-Lipschitz} \\
&\emod{a} \lfn(0) .
\end{align*}
Now suppose that $x \in \Z_{\geq 0}$ is such that $\lfn(x + a) \equiv_a \lfn(x)$. Then
\begin{align*}
\lfn(x + 1 + a) &\emod{a} (x + 1 + a)\lfn(x + a) + \rfn(x + 1 + a) \\
&\cmod{a} (x + 1)\lfn(x + a) + \rfn(x + 1 + a) \eqcom{because $|\lfn(x + a)| \leq 1$} \\
&\cmod{a} (x + 1)\lfn(x) + \rfn(x + 1 + a) \eqcom{by the inductive hyp.\ on $x$} \\
&\cmod{a} (x + 1)\lfn(x) + \rfn(x + 1) \eqcom{because $\rfn$ is $1$-Lipschitz} \\
&\emod{a} \lfn(x + 1) .
\end{align*}

We now know that $\lfn(x + a) \equiv_a \lfn(x)$ for all $x \in \Z_{\geq 0}$, and it remains to extend this to $x \in \Zp$, which we may achieve by continuity. Specifically, the function
\bea
\Zp &\ra& A \\
x &\mapsto& \lfn(x + a) - \lfn(x)
\eea
maps $\Z_{\geq 0}$ into the set $V = \{\alpha \in A \sat |\alpha| \leq |a|\}$ and therefore, being continuous, maps $\Zp$ into $V$ as well because $V$ is closed in $A$.
\end{proof}

We apply the above to the functions $\gfn{r}$ with $r \in 1 + p\Zp$.

\begin{prop}
Suppose that $r \in \Cp$ satisfies $|r - 1| \leq (1/3)^{1/3}$. (This condition is satisfied, in particular, if $r \in 1 + p\Zp$, since $(1/3)^{1/3} > 1/p$ no matter what prime $p$ is.) Then $\gfn{r}$ is $1$-Lipschitz.
\end{prop}

\begin{proof}
Recall that $\gfn{r} = \calS^{-1}(g_r)$ where $g_r(x) = r^x$. By Proposition~\ref{lip prop}, the statement to be proven is equivalent to the statement that $g_r$ is $1$-Lipschitz. For that, note that if $x,y \in \Zp$, and if $n \geq 0$, then
\[ \left|\ch{x}{n} - \ch{y}{n}\right| \leq n|x - y| .\]
This is a weak form of the lemma in \cite[Chap.~5, Sect.~1.5]{robert:p-adic}. Hence,
\begin{align*}
|g_r(x) - g_r(y)| &= \left|\sum_{n = 0}^\infty (r - 1)^n\left(\ch{x}{n} - \ch{y}{n}\right)\right| \\
&= \left|\sum_{n = 1}^\infty (r - 1)^n\left(\ch{x}{n} - \ch{y}{n}\right)\right| \eqcomb{the $n = 0$ term vanishes} \\
&\leq \max_{n \in \Z_{\geq 1}} |r - 1|^n n|x - y| \\
&= |x - y|\max_{n \in \Z_{\geq 1}} n|r - 1|^n .
\end{align*}
We claim that $\max_{n \in \Z_{\geq 1}} n|r - 1|^n \leq 1$. Indeed, this inequality holds if and only if $n|r - 1|^n \leq 1$ for all $n \in \Z_{\geq 1}$, if and only if $|r - 1| \leq (1/n)^{1/n}$ for all $n \in \Z_{\geq 1}$, if and only if
\[ |r - 1| \leq \min_{n \in \Z_{\geq 1}} (1/n)^{1/n} = (1/3)^{1/3} .\]
\end{proof}

\section{Generating functions}

If $\lfn \in \cfns{\Zp}{\Cp}$, let
\begin{align*}
\ogen{\lfn}(\gfvar) &= \sum_{n = 0}^\infty \lfn(n)\gfvar^n ,\\*
\egen{\lfn}(\gfvar) &= \sum_{n = 0}^\infty \lfn(n)\frac{\gfvar^n}{n!} ,
\end{align*}
the ordinary and exponential generating functions, respectively, of $(\lfn(n))_{n \geq 0}$. We also let
\[ \agen{\lfn}(\gfvar) = \ogen{\widehat{\lfn}}(-t) = \sum_{n = 0}^\infty \widehat{\lfn}(n)(-t)^n = \sum_{n = 0}^\infty (-1)^n \lfn(-1 - n)t^n .\]

\begin{prop} \label{calS and gen funcs}
Let $\lfn \in \cfns{\Zp}{\Cp}$ and $\rfn = \calS(\lfn)$. Then
\begin{align*}
\egen{\lfn}(\gfvar) &= \frac{\egen{\rfn}(\gfvar)}{1 - \gfvar} \\
\text{and}\quad \ogen{\widehat{\lfn}}'(\gfvar) + \ogen{\widehat{\lfn}}(\gfvar) &= \ogen{\widehat{\rfn}}(\gfvar) .
\end{align*}
The second equality is equivalent to $\agen{\lfn}(t) - \agen{\lfn}'(t) = \agen{\rfn}(t)$.
\end{prop}

\begin{proof}
Note that the equality $\calS(\lfn) = \rfn$ means that $\lfn(x) = x\lfn(x - 1) + \rfn(x)$ for all $x \in \Zp$. Then
\begin{align*}
(1 - \gfvar)\egen{\lfn}(\gfvar) &= \sum_{n = 0}^\infty \lfn(n)\,\frac{\gfvar^n}{n!} - \sum_{n = 0}^\infty \lfn(n)\,\frac{\gfvar^{n + 1}}{n!} \\
&= \sum_{n = 0}^\infty \lfn(n)\,\frac{\gfvar^n}{n!} - \sum_{n = 1}^\infty \lfn(n - 1)\,\frac{\gfvar^n}{(n - 1)!} \\
&= \sum_{n = 0}^\infty (\lfn(n) - n\lfn(n - 1))\frac{\gfvar^n}{n!} \\
&= \sum_{n = 0}^\infty \rfn(n)\,\frac{\gfvar^n}{n!} \\
&= \egen{\rfn}(\gfvar) ,
\end{align*}
and
\begin{align*}
\ogen{\lfn}'(\gfvar) + \ogen{\lfn}(\gfvar) &= \sum_{n = 0}^\infty ((n + 1)\widehat{\lfn}(n + 1) + \widehat{\lfn}(n))\gfvar^n \\
&= \sum_{n = 0}^\infty ((n + 1)\lfn(-2 - n) + \lfn(-1 - n))\gfvar^n \\
&= \sum_{n = 0}^\infty \rfn(-1 - n)\gfvar^n \\
&= \ogen{\widehat{\rfn}}(\gfvar) .
\end{align*}
\end{proof}

For example, let us apply the proposition to the functions $\gfn{r} = \calS^{-1}(g_r)$ appearing in Section~\ref{sec: p-adic inc}.

\begin{cor}
If $r \in \lball{1}{1}[\Cp]$, then
\begin{align*}
\egen{\gfn{r}}(\gfvar) &= \frac{\exp(r\gfvar)}{1 - \gfvar} \\
\text{and}\quad \ogen{\widehat{f}_r}'(\gfvar) + \ogen{\widehat{f}_r}(\gfvar) &= \frac{1}{r - t} .
\end{align*}
\end{cor}

\begin{proof}
The equalities follow immediately from the proposition and the straightforward observations that
\[ \egen{g_r}(\gfvar) = \exp(rt) ,\quad \ogen{\widehat{g}_r}(\gfvar) = \frac{1}{r - t} .\]
\end{proof}

\section{Self-adjointness of $\calS$} \label{sec: self-adj}

Recall that a linear functional $\mu : \cfns{\Zp}{\Cp} \ra \Cp$ is said to be \emph{bounded} if there is $C > 0$ such that $|\mu(\phi)| \leq C\|\phi\|$ for all $\phi \in \cfns{\Zp}{\Cp}$. In the following, we use the term \emph{measure} to mean a bounded linear functional $\mu : \cfns{\Zp}{\Cp} \ra \Cp$. Our approach is similar to that of \cite[Chap.~4, Sec.~1]{lang:cyc}, except that measures there are assumed to have values in $\{x \in \Cp \sat |x| \leq 1\}$. We do not impose this restriction.

If $\mu$ is a measure, then the power series $G(\gfvar) = \sum_{n = 0}^\infty \mu(\beta_n)\gfvar^n$ has bounded coefficients. (Recall that $\beta_n(x) = \ch{x}{n}$.) Conversely, a power series $G(\gfvar) = \sum_{n = 0}^\infty b_n\gfvar^n$ with bounded coefficients defines a measure $\mu$ satisfying
\[ \mu(\phi) = \sum_{n = 0}^\infty a_n b_n ,\]
where $\phi \in \cfns{\Zp}{\Cp}$ has Mahler coefficients $(a_n)_{n \geq 0}$. The measure $\mu$ associated to $G$ in this way will be denoted $\mu(G)$.

If $\mu$ is a measure and $\phi \in \cfns{\Zp}{\Cp}$, we denote $\mu(\phi)$ by $\int_{\Zp} \phi\,d\mu$.

In the special case where $\phi : x \mapsto (1 + z)^x$, where $z$ is an element of $\Cp$ with $|z| < 1$, we have
\beq \label{int (1 + z)^x}
\int_{\Zp} (1 + z)^x\,d\mu(G)(x) = G(z) .
\eeq
This is immediate, since the function $\phi$ has $n$th Mahler coefficient $a_n = z^n$, so if $G(\gfvar) = \sum_{n = 0}^\infty b_n\gfvar^n$, then
\[ \sum_{n = 0}^\infty a_n b_n = \sum_{n = 0}^\infty z^n b_n = G(z) .\]

We now define a $\Cp$-valued pairing on the $\Cp$-vector space $\cfns{\Zp}{\Cp}$ by
\[ \pair{\phi}{\psi} = \int_{\Zp} \phi\,d\mu(H_\psi) .\]

The following lemma will be used several times.

\begin{lemma} \label{a_k b_l lemma}
Let $(a_k)_{k \geq 0}$ and $(b_l)_{l \geq 0}$ be sequences in $\Cp$ that converge to zero. Then for any $\eps > 0$, there are $M,N \geq 0$ such that if either $k \geq M$ or $l \geq N$, then $|a_k b_l| \leq \eps$.
\end{lemma}

\begin{proof}
Fix $\eps > 0$. First, choose $L \geq 0$ such that $|b_l| \leq 1$ for all $l \geq L$, possible because $b_l \ra 0$ by assumption. Then, for each $l < L$, choose $M_l \geq 0$ such that $|a_k|\,|b_l| \leq \eps$ for all $k \geq M_l$, which is again possible because $a_k \ra 0$. Next, choose $\widetilde{M}$ such that $|a_k| \leq \eps$ for $k \geq \widetilde{M}$. Now let $M = \max\{\widetilde{M},M_0,\ldots,M_{L - 1}\}$. Suppose that $k \geq M$, and let $l \geq 0$ be arbitrary. If $l < L$, then because $k \geq M \geq M_l$, we have $|a_k|\,|b_l| \leq \eps$. Otherwise, if $l \geq L$, then $|b_l| \leq 1$, while $|a_k| \leq \eps$ because $k \geq \widetilde{M}$, so $|a_k|\,|b_l| \leq \eps$.

In summary, $M$ has the property that $|a_k b_l| \leq \eps$ whenever $k \geq M$, irrespective of the value of $l$. One may prove similarly the existence of an $N$ such that $|a_k b_l| \leq \eps$ whenever $l \geq N$, irrespective of the value of $k$.
\end{proof}

\begin{prop} \label{bilinear form prop}
The pairing $\pair{\cdot}{\cdot}$ is a non-degenerate symmetric bilinear form.
\end{prop}

\begin{proof}
Bilinearity is clear. For symmetry, we observe that if $l$ is a non-negative integer, then $\ch{-1 - x}{l} = (-1)^l\ch{x + l}{l}$. Now, if $\phi,\psi \in \cfns{\Zp}{\Cp}$ have Mahler series
\[ \phi(x) = \sum_{k = 0}^\infty a_k\ch{x}{k} ,\quad \psi(x) = \sum_{l = 0}^\infty b_l\ch{x}{l} ,\]
then
\[ \agen{\psi}(\gfvar) = \sum_{k = 0}^\infty (-1)^k\psi(-1 - k)\gfvar^k = \sum_{k = 0}^\infty (-1)^k\sum_{l = 0}^\infty b_l\ch{-1 - k}{l}\gfvar^k ,\]
so
\begin{align}
\pair{\phi}{\psi} &= \int_{\Zp} \phi\,d\mu(H_\psi) \label{symm: step 1} \\
&= \sum_{k = 0}^\infty a_k(-1)^k\sum_{l = 0}^\infty b_l\ch{-1 - k}{l} \\
&= \sum_{k = 0}^\infty \sum_{l = 0}^\infty (-1)^{k + l}a_k b_l\ch{k + l}{l} \label{symm: step 3} \\
&= \sum_{k = 0}^\infty \sum_{l = 0}^\infty (-1)^{k + l}a_k b_l\ch{k + l}{k} \nonumber \\
&= \sum_{l = 0}^\infty \sum_{k = 0}^\infty (-1)^{k + l}a_k b_l\ch{k + l}{k} \eqcomb{see below} \label{symm: interchange step} \\
&= \sum_{l = 0}^\infty \sum_{k = 0}^\infty (-1)^{l + k}b_l a_k\ch{l + k}{k} , \nonumber
\end{align}
and this last expression is equal to $\pair{\psi}{\phi}$ by a reversal of steps (\ref{symm: step 1})--(\ref{symm: step 3}) with the roles of $\phi$ and $\psi$ interchanged.

We show that the interchanging of the summation signs at (\ref{symm: interchange step}) is permissible. By \cite[Chap.~2, Sect.~1.2, Cor.]{robert:p-adic}, it is enough to show that, for any $\eps > 0$, there are only finitely many pairs $(k,l)$ such that $|\eta(k,l)| > \eps$, where $\eta(k,l) = (-1)^{k + l}a_k b_l\ch{k + l}{k}$. Of course, $|\eta(k,l)| \leq |a_k|\,|b_l|$, so it is enough to show that there are $M,N \geq 0$ such that
\begin{enumerate}[label=(\roman*)]
\item for all $k \geq M$ and all $l \geq 0$, $|a_k|\,|b_l| \leq \eps$, and
\item for all $k \geq 0$ and all $l \geq N$, $|a_k|\,|b_l| \leq \eps$.
\end{enumerate}
For this, we simply invoke Lemma~\ref{a_k b_l lemma}.

Finally, we show that $\pair{\cdot}{\cdot}$ is non-degenerate. Fix $\psi \in \cfns{\Zp}{\Cp}$, and suppose that $\pair{\phi}{\psi} = 0$ for all $\phi \in \cfns{\Zp}{\Cp}$. If $\phi(x) = \ch{x}{n}$, then $\pair{\phi}{\psi}$ is the coefficient of $\gfvar^n$ in $\agen{\psi}(\gfvar)$, which is $(-1)^n\psi(-1 - n)$, so $n \geq 0$ being arbitrary, we see that $\psi(-1 - n) = 0$ for all $n \geq 0$, and then $\psi = 0$ by continuity.
\end{proof}

In fact, we have the following more symmetric version of the double sum for $\pair{\phi}{\psi}$ appearing in the proof of Proposition~\ref{bilinear form prop}. It is more symmetric because the order of summation does not privilege either direction, $k$ or $l$, in the lattice of pairs $(k,l)$; rather, it runs through the lattice diagonally.

\begin{prop} \label{symm version of pairing}
Suppose that $\phi$ and $\psi$ have Mahler coefficients $(a_k)_{k \geq 0}$ and $(b_l)_{l \geq 0}$ respectively. Then
\[ \pair{\phi}{\psi} = \sum_{n = 0}^\infty (-1)^n \sum_{k = 0}^n \ch{n}{k}a_k b_{n - k} .\]
\end{prop}

\begin{proof}
By (\ref{symm: step 3}), $\pair{\phi}{\psi} = \sum_{k = 0}^\infty \sum_{l = 0}^\infty \eta(k,l)$, where $\eta(k,l) = (-1)^{k + l}\ch{k + l}{l}a_k b_l$. Let $(\alpha_i)_{i \geq 0}$ be the sequence
\[ \begin{array}{l}
(\eta(0,0), \\
\phantom{(}\eta(0,1),\eta(1,0), \\
\phantom{(}\eta(0,2),\eta(1,1),\eta(2,0), \\
\phantom{(}\eta(0,3),\eta(1,2),\eta(2,1),\eta(3,0), \\
\phantom{(}\ldots) .
\end{array} \]
We claim that this sequence converges to zero. Let $\eps > 0$. We again use Lemma~\ref{a_k b_l lemma} to obtain $M,N \geq 0$ such that if either $k \geq M$ or $l \geq N$, then $|a_k b_l| \leq \eps$. Hence, if $k + l \geq M + N$, then either $k \geq M$ or $l \geq N$, so $|\eta(k,l)| \leq \eps$. This establishes the claim that $\alpha_0,\alpha_1,\alpha_2,\ldots$ converges to zero. But then \cite[Chap.~2, Sect.~1.2, Prop.]{robert:p-adic} implies that
\[ \sum_{i = 0}^\infty \alpha_i = \sum_{k = 0}^\infty \sum_{l = 0}^\infty \eta(k,l) ,\]
\ie
\begin{align*}
\sum_{n = 0}^\infty (-1)^n \sum_{k = 0}^n \ch{n}{k}a_k b_{n - k} &= \sum_{k = 0}^\infty \sum_{l = 0}^\infty \eta(k,l) \\
&= \pair{\phi}{\psi} .
\end{align*}
\end{proof}

\begin{lemma} \label{measure assoc to deriv}
If $G \in \Cp\pwr{\gfvar}$ has bounded coefficients, and if $\phi \in \cfns{\Zp}{\Cp}$, then
\[ \int_{\Zp} \phi\,d\mu(G') = \int_{\Zp} x\phi(x - 1)\,d\mu(G)(x) .\]
\end{lemma}

\begin{proof}
Suppose that
\[ G(\gfvar) = \sum_{n = 0}^\infty b_n\gfvar^n \quad\text{and}\quad \phi(x) = \sum_{n = 0}^\infty a_n\ch{x}{n} .\]
Then $G'(\gfvar) = \sum_{n = 0}^\infty (n + 1)b_{n + 1}\gfvar^n$, so
\beq \label{measure assoc to deriv: eqn}
\int_{\Zp} \phi\,d\mu(G') = \sum_{n = 0}^\infty a_n(n + 1)b_{n + 1} = \sum_{n = 1}^\infty na_{n - 1}b_n .
\eeq
But
\begin{align*}
\sum_{n = 1}^\infty na_{n - 1}\ch{x}{n} &= \sum_{n = 1}^\infty a_{n - 1}x\ch{x - 1}{n - 1} \eqcomb{a standard binomial identity} \\
&= x\sum_{n = 0}^\infty a_n\ch{x - 1}{n} \\
&= x\phi(x - 1) ,
\end{align*}
so (\ref{measure assoc to deriv: eqn}) shows that $\int_{\Zp} \phi\,d\mu(G') = \int_{\Zp} x\phi(x - 1)\,d\mu(G)(x)$.
\end{proof}

\begin{theorem} \label{pairing is self-adj}
The operator $\calS$ on $\cfns{\Zp}{\Cp}$ is self-adjoint with respect to $\pair{\cdot}{\cdot}$, that is,
\[ \pair{\calS(\phi)}{\psi} = \pair{\phi}{\calS(\psi)} \]
for all $\phi,\psi \in \cfns{\Zp}{\Cp}$.
\end{theorem}

\begin{proof}
\begin{align*}
\pair{\calS(\phi)}{\psi} &= \int_{\Zp} \calS(\phi)\,d\mu(H_\psi) \\
&= \int_{\Zp} (\phi(x) - x\phi(x - 1))\,d\mu(H_\psi)(x) \\
&= \int_{\Zp} \phi(x)\,d\mu(H_\psi)(x) - \int_{\Zp} x\phi(x - 1)\,d\mu(H_\psi)(x) \\
&= \int_{\Zp} \phi\,d\mu(H_\psi) - \int_{\Zp} \phi\,d\mu(H_\psi') \eqcom{by Lemma~\ref{measure assoc to deriv}} \\
&= \int_{\Zp} \phi\,d\mu(H_\psi - H_\psi') \\
&= \int_{\Zp} \phi\,d\mu(H_{\calS(\psi)}) \eqcom{by Proposition~\ref{calS and gen funcs}} \\
&= \pair{\phi}{\calS(\psi)} .
\end{align*}
\end{proof}

\begin{cor} \label{inc gamma adjointness}
If $r,s \in \lball{1}{1}[\Cp]$, then $\pair{g_r}{\gfn{s}} = \pair{\gfn{r}}{g_s}$, where the functions $\gfn{r},g_r$ are as in Section~\ref{sec: p-adic inc}.
\end{cor}

\begin{proof}
This follows immediately from the fact that $\calS(\gfn{r}) = g_r$.
\end{proof}

\begin{cor}
If $r,s \in \lball{1}{1}[\Cp]$, and if $(\gamma_{r,n})_{n \geq 0}$ is the sequence of Mahler coefficients of $\gfn{r}$, then
\[ \sum_{n = 0}^\infty (1 - r)^n \gfn{s}(-1 - n) = \sum_{n = 0}^\infty (-1)^n \gamma_{r,n}s^{-1 - n} .\]
\end{cor}

\begin{proof}
On the one hand,
\begin{align*}
\pair{g_r}{\gfn{s}} &= \int_{\Zp} r^x\,d\mu(H_{\gfn{s}})(x) \\
&= H_{\gfn{s}}(r - 1) \eqcom{by (\ref{int (1 + z)^x})} \\
&= \sum_{n = 0}^\infty (1 - r)^n \gfn{s}(-1 - n) .
\end{align*}
On the other hand,
\[ \pair{\gfn{r}}{g_s} = \int_{\Zp} \gfn{r}\,d\mu(H_{g_s}) = \sum_{n = 0}^\infty \gamma_{r,n}(-1)^n s^{-1 - n} .\]
Now use Corollary~\ref{inc gamma adjointness}.
\end{proof}

More generally, we have the following.

\begin{cor} \label{gamma r mahl psi corollary}
If $r \in \lball{1}{1}[\Cp]$ and $(\gamma_{r,n})_{n \geq 0}$ is the sequence of Mahler coefficients of $\gfn{r}$, then for any $\psi \in \cfns{\Zp}{\Cp}$,
\begin{align*}
\sum_{n = 0}^\infty (-1)^n\calS(\psi)(-1 - n)\gamma_{r,n} &= H_\psi(r - 1) \\
&= \sum_{n = 0}^\infty \psi(-1 - n)(1 - r)^n .
\end{align*}
\end{cor}

\begin{proof}
Taking $\phi = \gfn{r}$ in the theorem gives $\pair{g_r}{\psi} = \pair{\gfn{r}}{\calS(\psi)}$. But
\[ \pair{g_r}{\psi} = \int_{\Zp} r^x\,d\mu(H_\psi)(x) = H_\psi(r - 1) \]
by (\ref{int (1 + z)^x}), and this is equal to $\sum_{n = 0}^\infty \psi(-1 - n)(1 - r)^n$. On the other hand,
\[ \pair{\gfn{r}}{\calS(\psi)} = \int_{\Zp} \gfn{r}\,d\mu(H_{\calS(\psi)}) = \sum_{n = 0}^\infty (-1)^n\calS(\psi)(-1 - n)\gamma_{r,n} .\]
\end{proof}

\bex
If in Corollary~\ref{gamma r mahl psi corollary} we take $\psi$ to be the constant function $\one$, then $H_\psi(\gfvar) = 1/(1 + \gfvar)$ and $\calS(\psi)(x) = 1 - x$, so we have
\[ \sum_{n = 0}^\infty (-1)^n(n + 2)\gamma_{r,n} = \frac{1}{r} .\]
Note that the case $r = 1$ says that $\sum_{n = 0}^\infty (-1)^n(n + 2)\cdot n! = 1$, since $\gfn{1}(x) = \qf(x) = \sum_{n = 0}^\infty n!\ch{x}{n}$.
\eex

\section{Convolution of Mahler coefficients} \label{sec: conv of mahl}

Let $\decayegf$ denote the set of power series $G(\gfvar) = \sum_{n = 0}^\infty a_n\tfrac{\gfvar^n}{n!} \in \Cp\pwr{\gfvar}$ such that $a_n \to 0$.

\begin{prop}
$\decayegf$ is a $\Cp$-subalgebra of $\Cp\pwr{\gfvar}$.
\end{prop}

\begin{proof}
It is obvious that $\decayegf$ is a $\Cp$-subspace. Now suppose that $G(\gfvar) = \sum_{n = 0}^\infty a_n\tfrac{\gfvar^n}{n!}$ and $H(\gfvar) = \sum_{n = 0}^\infty b_n\tfrac{\gfvar^n}{n!}$, where $a_n \to 0$ and $b_n \to 0$. Then
\[ G(\gfvar)H(\gfvar) = \sum_{n = 0}^\infty \left(\sum_{k = 0}^n \ch{n}{k}a_k b_{n - k}\right)\frac{\gfvar^n}{n!} .\]
We show that the sequence of elements
\[ \sum_{k = 0}^n \ch{n}{k}a_k b_{n - k} \]
converges to zero as $n \to \infty$. Let $\eps > 0$. By Lemma~\ref{a_k b_l lemma} again, we may choose $M,N \geq 0$ such that if either $k \geq M$ or $l \geq N$, then $|a_k b_l| \leq \eps$. If $n \geq M + N$ and $0 \leq k \leq n$, then either $k \geq M$ or $n - k \geq N$, so $|a_k b_{n - k}| \leq \eps$, and then
\[ \left|\sum_{k = 0}^n \ch{n}{k}a_k b_{n - k}\right| \leq \eps .\]
\end{proof}

Now consider the mutually inverse maps
\beal
\cfns{\Zp}{\Cp} &\ra& \decayegf \nonumber \\*
\phi &\mapsto& \sum_{n = 0}^\infty (\nabla^n\phi)(0)\frac{\gfvar^n}{n!} , \label{cfns decay corr 1} \\[2ex]
\decayegf &\ra& \cfns{\Zp}{\Cp} \nonumber \\*
G &\mapsto& \sum_{n = 0}^\infty G^{(n)}(0)\beta_n . \label{cfns decay corr 2}
\eeal
They are $\Cp$-linear maps but not algebra homomorphisms if $\cfns{\Zp}{\Cp}$ has the usual product. We transfer the product on $\decayegf \con \C\pwr{\gfvar}$ to a new product $\star$ on $\cfns{\Zp}{\Cp}$ via the above maps. Concretely, if $\phi$ and $\psi$ have Mahler coefficients $(a_k)_{k \geq 0}$ and $(b_l)_{l \geq 0}$ respectively, then
\beq \label{binom mahl conv}
\phi \star \psi = \sum_{n = 0}^\infty \left(\sum_{k = 0}^n \ch{n}{k}a_k b_{n - k}\right)\beta_n ,
\eeq
so $\star$ is a binomial convolution of the Mahler coefficients. Note that the identity element of $(\cfns{\Zp}{\Cp},+,\star)$ is the constant function $\one$.

\begin{remark}
By \cite[Chap.~4, Sect.~1.1, Comment (2)]{robert:p-adic}, the power series corresponding to a continuous function $\phi \in \cfns{\Zp}{\Cp}$ via (\ref{cfns decay corr 1}) and (\ref{cfns decay corr 2}) is $\exp(-t)\egen{\phi}(t)$.
\end{remark}

Note that, because $(\decayegf,+,\cdot)$ is a commutative $\Cp$-algebra with no zero divisors, the same is true of $(\cfns{\Zp}{\Cp},+,\star)$.

We will also require the following notation. If $\phi \in \cfns{\Zp}{\Cp}$ and $m \in \Z_{\geq 0}$, let
\beq \label{star power non-neg}
\phi^{\star m} = \underbrace{\phi \star \cdots \star \phi}_m .
\eeq
If $\phi$ is a unit in $\cfns{\Zp}{\Cp}$ with respect to the product operation $\star$, then for $m < 0$ we let
\beq \label{star power neg}
\phi^{\star m} = (\phi^{-1})^{\star(-m)} ,
\eeq
where $\phi^{-1}$ here denotes the inverse of $\phi$ with respect to $\star$.

\begin{prop} \label{star and pairing}
If $\phi,\psi \in \cfns{\Zp}{\Cp}$, then $\pair{\phi}{\psi} = (\phi \star \psi)(-1)$.
\end{prop}

\begin{proof}
Let the Mahler coefficients of $\phi$ and $\psi$ be $(a_k)_{k \geq 0}$ and $(b_l)_{l \geq 0}$ respectively. Then by (\ref{binom mahl conv}),
\[ (\phi \star \psi)(-1) = \sum_{n = 0}^\infty (-1)^n \sum_{k = 0}^n \ch{n}{k}a_k b_{n - k} ,\]
because $\beta_n(-1) = \ch{-1}{n} = (-1)^n$. But this is equal to $\pair{\phi}{\psi}$ by Proposition~\ref{symm version of pairing}.
\end{proof}

\begin{prop} \label{pairing as integral of product}
If $\phi,\psi \in \cfns{\Zp}{\Cp}$, then $\pair{\phi}{\psi} = \int_{\Zp} (\phi \star \psi)\,d\mu(H_\one)$, where $\one$ is the function with constant value $1$.
\end{prop}

\begin{proof}
\begin{align*}
\pair{\phi}{\psi} &= (\phi \star \psi)(-1) \eqcom{by Proposition~\ref{star and pairing}} \\
&= ((\phi \star \psi) \star \one)(-1) \\
&= \pair{\phi \star \psi}{\one} \eqcom{by the same proposition} \\
&= \int_{\Zp} (\phi \star \psi)\,d\mu(H_\one) \eqcom{by definition.}
\end{align*}
\end{proof}

\begin{theorem} \label{calS and star}
If $\phi,\psi \in \cfns{\Zp}{\Cp}$, then $\calS(\phi) \star \psi = \phi \star \calS(\psi)$.
\end{theorem}

\begin{proof}
Suppose that $\phi = \sum_{k = 0}^\infty a_k\beta_k$ and $\psi = \sum_{l = 0}^\infty b_l\beta_l$. On the one hand,
\[ \calS(\phi) \star \psi = \sum_{n = 0}^\infty \left(\sum_{k = 0}^n \ch{n}{k}(a_k - ka_{k - 1})b_{n - k}\right)\beta_n \]
by Proposition~\ref{calS of mahler}, while on the other hand
\[ \phi \star \calS(\psi) = \sum_{n = 0}^\infty \left(\sum_{k = 0}^n \ch{n}{k}a_{n - k}(b_k - kb_{k - 1})\right)\beta_n .\]
Note that we have set $a_{-1} = b_{-1} = 0$. We have to show that
\[ \sum_{k = 0}^n \ch{n}{k} (a_k - ka_{k - 1})b_{n - k} = \sum_{k = 0}^n \ch{n}{k}a_{n - k}(b_k - kb_{k - 1}) .\]
Since $\sum_{k = 0}^n \ch{n}{k}a_k b_{n - k} = \sum_{k = 0}^n \ch{n}{k} a_{n - k}b_k$, it remains to show that
\[ \sum_{k = 0}^n \ch{n}{k}ka_{k - 1}b_{n - k} = \sum_{k = 0}^n \ch{n}{k}ka_{n - k}b_{k - 1} .\]
Now,
\begin{align*}
\sum_{k = 0}^n \ch{n}{k}ka_{k - 1}b_{n - k} &= \sum_{k = 0}^n \ch{n}{k}(n - k)a_{n - k - 1}b_k \eqcomb{$k \mapsto n - k$} \\
&= \sum_{k = 0}^{n - 1} \ch{n}{k}(n - k)a_{n - k - 1}b_k \eqcomb{the $k = n$ term is $0$} \\
&= \sum_{k = 1}^n \ch{n}{k - 1}(n - k + 1)a_{n - k}b_{k - 1} \eqcomb{reindexing} \\
&= \sum_{k = 1}^n \ch{n}{k}ka_{n - k}b_{k - 1} \eqcomb{a binomial identity} \\
&= \sum_{k = 0}^n \ch{n}{k}ka_{n - k}b_{k - 1} \eqcomb{the $k = 0$ term is $0$}[,]
\end{align*}
as desired.
\end{proof}

\begin{cor} \label{calS corr to 1 - beta_1}
If $\psi \in \cfns{\Zp}{\Cp}$, then
\begin{align*}
\calS(\psi) &= (\one - \beta_1) \star \psi ,\\
\text{and}\quad \calS^{-1}(\psi) &= \qf \star \psi .
\end{align*}
Therefore, $\calS^m(\psi) = (\one - \beta_1)^{\star m} \star \psi$ for all $m \in \Z$.
\end{cor}

\begin{proof}
\begin{align*}
\calS(\psi) &= \one \star \calS(\psi) = \calS(\one) \star \psi = (\one - \beta_1) \star \psi ,\\
\text{and}\quad \calS^{-1}(\psi) &= \one \star \calS^{-1}(\psi) = \calS^{-1}(\one) \star \psi = \qf \star \psi .
\end{align*}
For the last assertion, observe that $\qf \star (\one - \beta_1) = \calS^{-1}(\one) \star \calS(\one) = \one$.
\end{proof}

For example, if $r \in \lball{1}{1}[\Cp]$, then
\[ \gfn{r} = \calS^{-1}(g_r) = \qf \star g_r = \gfn{1} \star g_r .\]

\begin{cor}
Under the bijections in (\ref{cfns decay corr 1}) and (\ref{cfns decay corr 2}), the map $\calS : \cfns{\Zp}{\Cp} \ra \cfns{\Zp}{\Cp}$ corresponds to the map
\bea
\decayegf &\ra& \decayegf \\*
G(t) &\mapsto& (1 - \gfvar)G(t) .
\eea
\end{cor}

\begin{proof}
The function $\one - \beta_1 \in \cfns{\Zp}{\Cp}$ corresponds to the element $1 - \gfvar \in \decayegf$. Now use Corollary~\ref{calS corr to 1 - beta_1}.
\end{proof}

\begin{cor}
A function $\phi \in \cfns{\Zp}{\Cp}$ is a unit with respect to $\star$ if and only if $\calS(\phi)$ is. Thus, $\calS$ restricts to a permutation of the unit group of the algebra $(\cfns{\Zp}{\Cp},+,\star)$.
\end{cor}

\begin{proof}
If $\calS(\phi) \star \psi = \one$, then $\phi \star \calS(\psi) = \one$. Conversely, if $\phi \star \psi = \one$, then $\calS(\phi) \star \calS^{-1}(\psi) = \one$.
\end{proof}

\begin{cor} \label{calS has infinite order}
If $\psi \in \cfns{\Zp}{\Cp}$ and $\calS^m(\psi) = \psi$ for some non-zero $m \in \Z$, then $\psi = 0$.
\end{cor}

\begin{proof}
By Corollary~\ref{calS corr to 1 - beta_1}, the equality $\calS^m(\psi) = \psi$ says that $(\one - \beta_1)^{\star m} \star \psi = \psi$, \ie $((\one - \beta_1)^{\star m} - \one) \star \psi = 0$. But $\one - \beta_1$ corresponds to $1 - \gfvar \in \decayegf$, which has infinite multiplicative order, so because $m \not= 0$, $(\one - \beta)^{\star m} - \one \not= 0$. Hence, because $(\cfns{\Zp}{\Cp},+,\star)$ has no zero divisors, $\psi = 0$.
\end{proof}

\begin{remark}
Note that Theorem~\ref{calS and star} and Proposition~\ref{star and pairing} combine to give an alternative proof of Theorem~\ref{pairing is self-adj}, the self-adjointness of $\calS$ with respect to $\pair{\cdot}{\cdot}$.
\end{remark}

\section{An integral transform} \label{sec: integral transform}

Recall the notation $\phi^{\star m}$ defined in (\ref{star power non-neg}) and (\ref{star power neg}) for $\phi \in \cfns{\Zp}{\Cp}$ and $m \in \Z$. Given $\psi \in \cfns{\Zp}{\Cp}$, we may consider the function
\beal
\Z &\ra& \Cp \nonumber \\*
m &\mapsto& \int_{\Zp} (\one - \beta_1)^{\star m}\,d\mu(H_\psi) . \label{inttrans as integral: initial def}
\eeal
We show that this function, defined initially on $\Z$, can be extended to a locally analytic function $\Zp \ra \Cp$.

\subsection{Iterates of $\calS$} \label{iterates}

\begin{prop} \label{prop: iterates of calS}
If $\psi \in \cfns{\Zp}{\Cp}$, $m \in \Z_{\geq 0}$, and $x \in \Zp$, then
\[ \calS^m(\psi)(x) = \sum_{k = 0}^m (-1)^k\ch{m}{k}(x)_k\psi(x - k) .\]
\end{prop}

\begin{proof}
We proceed by induction on $m$, the case $m = 0$ being vacuous. Let $m \geq 0$, and assume that the statement is true for that $m$, \ie
\[ \calS^m(\psi)(y) = \sum_{k = 0}^m (-1)^k\ch{m}{k}(y)_k\psi(y - k) \eqcom{for all $y \in \Zp$.} \]
Then
\begin{align*}
\goinvisible &=\govisible \calS^{m + 1}(\psi)(x) \\
&= \calS(\calS^m(\psi))(x) \\
&= \calS\left(y \mapsto \sum_{k = 0}^m (-1)^k\ch{m}{k}(y)_k\psi(y - k)\right)(x) \\
&= \sum_{k = 0}^m (-1)^k\ch{m}{k}(x)_k\psi(x - k) - x\sum_{k = 0}^m (-1)^k\ch{m}{k}(x - 1)_k\psi(x - 1 - k) \\
&= \sum_{k = 0}^m (-1)^k\ch{m}{k}(x)_k\psi(x - k) + \sum_{k = 0}^m (-1)^{k + 1}\ch{m}{k}(x)_{k + 1}\psi(x - (k + 1)) \\
&= \sum_{k = 0}^m (-1)^k\ch{m}{k}(x)_k\psi(x - k) + \sum_{k = 1}^{m + 1} (-1)^k\ch{m}{k - 1}(x)_k\psi(x - k) \\
&= \psi(x) + \sum_{k = 1}^m (-1)^k\left(\ch{m}{k} + \ch{m}{k - 1}\right)(x)_k\psi(x - k) \\*
\goinvisible &=\govisible {} + (-1)^{m + 1}(x)_{m + 1}\psi(x - (m + 1)) \\
&= \sum_{k = 0}^{m + 1} (-1)^k\ch{m + 1}{k}(x)_k\psi(x - k) ,
\end{align*}
where in the last step we have used a standard binomial identity and reincorporated the $k = 0$ and $k = m + 1$ terms. The induction is complete.
\end{proof}

\begin{cor} \label{calS iterated: eval at -1}
If $\psi \in \cfns{\Zp}{\Cp}$ and $m \in \Z_{\geq 0}$, then
\[ \calS^m(\psi)(-1) = \sum_{k = 0}^m (m)_k\psi(-1 - k) .\]
\end{cor}

\begin{proof}
Take $x = -1$ in Proposition~\ref{prop: iterates of calS} and use the fact that $(-1)_k = (-1)^k\,k!$.
\end{proof}

\bex
Let $r \in \lball{1}{1}[\Cp]$. If we take $\psi = \gfn{r}$ in Corollary~\ref{calS iterated: eval at -1}, then we obtain
\[ \sum_{k = 0}^m (m)_k\gfn{r}(-1 - k) = \sum_{k = 0}^{m - 1} (m - 1)_k r^{-1 - k} \]
for all $m \in \Z_{\geq 1}$. Indeed, $\calS^m(\gfn{r}) = \calS^{m - 1}(g_r)$, and $g_r(x) = r^x$.

Specializing to $r = 1$ gives
\[ \sum_{k = 0}^m (m)_k\qf(-1 - k) = \qf(m - 1) \eqcomb{$m \geq 1$}[,] \]
because $\qf(x) = \sum_{n = 0}^\infty (x)_n$.
\eex

If $\psi \in \cfns{\Zp}{\Cp}$, then $\psi(\Zp)$ is bounded, so the sequence $(n!\,\psi(-1 - n))_{n \geq 0}$ converges to zero, and so we have a well-defined, continuous function
\beal
\inttrans{\psi} : \Zp &\ra& \Cp \nonumber \\*
\trvar &\mapsto& \sum_{n = 0}^\infty n!\,\psi(-1 - n)\ch{\trvar}{n} , \label{def of inttrans}
\eeal
that is, $\inttrans{\psi} = \sum_{n = 0}^\infty n!\,\psi(-1 - n)\beta_n$.

\begin{prop}
Let $\psi : \Zp \ra \Cp$ be continuous. Then $\inttrans{\psi}$ is locally analytic.
\end{prop}

\begin{proof}
We apply Proposition~\ref{amice variant} (in the $a_n'$ notation). By the boundedness of $\psi(\Zp)$, we may choose $C \in \R$ such that $v_p(\psi(-1 - n)) \geq C$ for all $n \geq 0$. Of course, there exist $\beta > -1/(p - 1)$ and $N \geq 1$ such that $C/n \geq \beta$ for all $n \geq N$. Hence, for such $n$,
\[ \frac{v_p(\psi(-1 - n))}{n} \geq \frac{C}{n} \geq \beta ,\]
and we are done.
\end{proof}

\begin{remark}
The map $\inttmap : \cfns{\Zp}{\Cp} \ra \lan{\Zp}{\Cp}$ is injective. Indeed, $\inttrans{\psi}$ determines its own Mahler coefficients and therefore the values $\psi(-1 - n)$ for all $n \geq 0$. By continuity, these values determine the function $\psi$.
\end{remark}

\begin{theorem} \label{power of S and inttrans}
Let $\psi \in \cfns{\Zp}{\Cp}$. The function
\bea
\Z &\ra& \Cp \\
m &\mapsto& \calS^m(\psi)(-1)
\eea
has a unique extension to a continuous function $\Zp \ra \Cp$. This continuous extension is in fact $\inttrans{\psi}$ and is therefore locally analytic on $\Zp$.
\end{theorem}

\begin{proof}
Uniqueness is clear by continuity, so we turn immediately to existence. By Corollary~\ref{calS iterated: eval at -1}, if $m \in \Z_{\geq 0}$, then
\[ \calS^m(\psi)(-1) = \sum_{n = 0}^m n!\,\psi(-1 - n)\ch{m}{n} = \sum_{n = 0}^\infty n!\,\psi(-1 - n)\ch{m}{n} = \inttrans{\psi}(m) .\]
In summary,
\beq \label{spwr first key step}
\calS^m(\psi)(-1) = \inttrans{\psi}(m) \eqcom{for all $\psi \in \cfns{\Zp}{\Cp}$ and all $m \in \Z_{\geq 0}$.}
\eeq
It remains to show that the equality $\calS^m(\psi)(-1) = \inttrans{\psi}(m)$ holds for $m \in \Z_{< 0}$ as well.

For the time being, we stay with $m \geq 0$, and in fact consider the situation where $m \geq 1$. In that case,
\begin{align*}
\calS^m(\psi)(-1) &= \calS^{m - 1}(\calS(\psi))(-1) ,\\
\eeie \inttrans{\psi}(m) &= \inttrans{\calS(\psi)}(m - 1)
\end{align*}
by (\ref{spwr first key step}) applied to both $(\psi,m)$ and $(\calS(\psi),m - 1)$. This being true for all $m \in \Z_{\geq 1}$, it follows by the continuity of $\inttrans{\psi}$ and $\inttrans{\calS(\psi)}$ that
\beq \label{spwr second key step}
\inttrans{\psi}(\trvar) = \inttrans{\calS(\psi)}(\trvar - 1) \eqcom{for all $\psi \in \cfns{\Zp}{\Cp}$ and $\trvar \in \Zp$.}
\eeq
Applying (\ref{spwr second key step}) with $\psi$ replaced by $\calS^{-1}(\psi)$, we obtain
\beq \label{spwr third key step}
\inttrans{\calS^{-1}(\psi)}(\trvar) = \inttrans{\psi}(\trvar - 1) \eqcom{for all $\psi \in \cfns{\Zp}{\Cp}$ and $\trvar \in \Zp$.}
\eeq
Iterating (\ref{spwr third key step}), \ie invoking induction, we arrive at
\beq \label{spwr fourth key step}
\inttrans{\calS^{-m}(\psi)}(\trvar) = \inttrans{\psi}(\trvar - m) \eqcom{for all $\psi \in \cfns{\Zp}{\Cp}$, $m \in \Z_{\geq 0}$, and $\trvar \in \Zp$.}
\eeq
Finally, taking $\trvar = 0$ in (\ref{spwr fourth key step}) gives $\inttrans{\calS^{-m}(\psi)}(0) = \inttrans{\psi}(-m)$, and since $\inttrans{\calS^{-m}(\psi)}(0) = \calS^{-m}(\psi)(-1)$ by definition, we have
\[ \calS^{-m}(\psi)(-1) = \inttrans{\psi}(-m) \eqcomb{$m \geq 0$}[,] \]
which provides the desired $m < 0$ case of (\ref{spwr first key step}).
\end{proof}

\begin{cor} \label{inttrans: calS and shiftu}
If $\psi \in \cfns{\Zp}{\Cp}$, $m \in \Z$, and $\trvar \in \Zp$, then
\[ \inttrans{\calS^m(\psi)}(\trvar) = \inttrans{\psi}(\trvar + m) .\]
\end{cor}

\begin{proof}
This is a corollary of the proof of the theorem. The case $m \leq 0$ is provided by (\ref{spwr fourth key step}). Now let $m \in \Z_{\geq 0}$, and apply (\ref{spwr fourth key step}) with $\psi$ replaced by $\calS^m(\psi)$. This gives
\[ \inttrans{\psi}(\trvar) = \inttrans{\calS^m(\psi)}(\trvar - m) \eqcom{for all $\trvar \in \Zp$.} \]
Replacing $\trvar$ by $\trvar + m$ gives the desired equality.
\end{proof}

For a continuous function $\psi : \Zp \ra \Cp$, let $\sumf\psi : \Zp \ra \Cp$ be its indefinite sum function, \ie the unique continuous function with the property that
\[ (\sumf\psi)(n) = \sum_{k = 0}^{n - 1} \psi(k) \]
for all $n \in \Z_{\geq 0}$. (See \cite[Chap.~4, Sect.~1.5]{robert:p-adic} and \cite[Chap.~4, Sect.~2.5]{robert:p-adic}.) Note that $(\sumf\psi)(0) = 0$ and that $\nabla\sumf\psi = \psi$.

\begin{prop} \label{sumf and inttmap}
If $\Psi \in \im(\inttmap)$, then $\calS^{-1}(\Psi) \in \im(\inttmap)$. Specifically,
\[ \calS^{-1}(\inttrans{\psi}) = \inttrans{-\sumf\psi} \]
for all $\psi \in \cfns{\Zp}{\Cp}$.
\end{prop}

\begin{proof}
We show that $\calS(\inttrans{\sumf\psi}) = \inttrans{-\psi}$, which is an equivalent assertion. Let $\phi = \sumf\psi$. Then because $\phi(0) = 0$, Proposition~\ref{calS of mahler} shows that
\begin{align*}
\calS(\inttrans{\sumf\psi}) &= \sum_{n = 0}^\infty (n!\,\phi(-1 - n) - n!\,\phi(-n))\beta_n \\
&= \sum_{n = 0}^\infty n!(\phi(-1 - n) - \phi(-n))\beta_n \\
&= \sum_{n = 0}^\infty n!(-(\nabla\phi)(-1 - n))\beta_n \\
&= \sum_{n = 0}^\infty n!(-\psi(-1 - n))\beta_n \\
&= \inttrans{-\psi} .
\end{align*}
\end{proof}

\subsection{Manifesting $\inttmap$ as an integral transform}

We now turn to our claim that the function $m \mapsto \int_{\Zp} (\one - \beta_1)^{\star m}\,d\mu(H_\psi)$ in (\ref{inttrans as integral: initial def}) can be extended to a locally analytic function on $\Zp$, showing at the same time that this extension is simply $\inttrans{\psi}$. We hence derive an integral-transform expression for $p$-adic incomplete $\Gamma$-functions.

Define
\defmap{\spwr}{\Zp}{\cfns{\Zp}{\Cp}}{\trvar}{\sum_{n = 0}^\infty (-1)^n\ch{\trvar}{n}n!\,\beta_n}[,]
where, as always, $\beta_n(x) = \ch{x}{n}$. Note that, for any fixed $s \in \Zp$, the function $\sum_{n = 0}^\infty (-1)^n\ch{\trvar}{n}n!\,\beta_n$ is indeed continuous, because its sequence of Mahler coefficients tends to zero.

Let $\|\cdot\|_M$ denote the norm on $\cfns{\Zp}{\Cp}$ defined by $\|\lfn\|_M = \max_{n \geq 0} |a_n|$, where $(a_n)_{n \geq 0}$ is the sequence of Mahler coefficients of $\lfn$.

\begin{prop} \label{T is uniformly continuous}
The map $\spwr$ is uniformly continuous with respect to $\|\cdot\|_M$.
\end{prop}

\begin{proof}
Let $\trvar_1,\trvar_2 \in \Zp$. Then
\[ \spwr(\trvar_1) - \spwr(\trvar_2) = \sum_{n = 1}^\infty (-1)^n\left(\ch{\trvar_1}{n} - \ch{\trvar_2}{n}\right)n!\,\beta_n ,\]
and for $n \geq 1$,
\beq \label{diff of T: bound on Mahler}
\left|(-1)^n\left(\ch{\trvar_1}{n} - \ch{\trvar_2}{n}\right)n!\right| \leq |\trvar_1 - \trvar_2| \cdot n|n!|
\eeq
by the lemma in \cite[Chap.~5, Sect.~1.5]{robert:p-adic}. (Actually, we have used only a weak form of the lemma, given in the remark that follows its proof.) An analysis of $n|n!|$ via the well-known fact that $v_p(n!) = (n - S_p(n))/(p - 1)$, where $S_p(n)$ is the $p$-adic digit sum of $n$, shows that $n|n!| \to 0$ as $n \to \infty$, so $n|n!|$ has a maximum $M$.

Now, let $\eps > 0$, and let $\delta = \eps/M$. If $|\trvar_1 - \trvar_2| < \delta$, then (\ref{diff of T: bound on Mahler}) shows that every Mahler coefficient of $\spwr(\trvar_1) - \spwr(\trvar_2)$ has absolute value at most $|\trvar_1 - \trvar_2|M < \delta M = \eps$, so $\|\spwr(\trvar_1) - \spwr(\trvar_2)\|_M < \eps$. Thus, $\spwr$ is uniformly continuous.
\end{proof}

\begin{prop} \label{T extends (1 - x)^star}
If $m \in \Z$, then $\spwr(m) = (\one - \beta_1)^{\star m}$.
\end{prop}

\begin{proof}
The continuous function $\one - \beta_1$ corresponds to the power series $1 - \gfvar$ under the bijections in (\ref{cfns decay corr 1}) and (\ref{cfns decay corr 2}), and if $m \geq 0$, then
\[ (1 - \gfvar)^m = \sum_{n = 0}^m \ch{m}{n}(-\gfvar)^n = \sum_{n = 0}^\infty (-1)^n\ch{m}{n}n!\,\frac{\gfvar^n}{n!} ,\]
which corresponds to $\spwr(m)$ under the same correspondence. For $m < 0$, we instead consider $\spwr(-m)$ where $m > 0$. Observe that $(\one - \beta_1)^{\star(-m)}$ corresponds to
\begin{align*}
(1 - \gfvar)^{-m} &= \sum_{n = 0}^\infty (-1)^n\ch{-m}{n}\gfvar^n \eqcomb{see \cite[Chap.~6, Sect.~3.2]{robert:p-adic}, for example} \\
&= \sum_{n = 0}^\infty (-1)^n\ch{-m}{n}n!\,\frac{\gfvar^n}{n!} ,
\end{align*}
and this corresponds to $\spwr(-m)$, as desired.
\end{proof}

If $\trvar \in \Zp$, we define $(\one - \beta_1)^{\star \trvar}$ to be $\spwr(\trvar)$. By Proposition~\ref{T extends (1 - x)^star}, this definition extends $(\one - \beta_1)^{\star \trvar}$ from $\trvar \in \Z$ to $\trvar \in \Zp$.

\begin{theorem} \label{inttrans as integral}
Let $\psi \in \cfns{\Zp}{\Cp}$, and let $\inttrans{\psi}$ be the locally analytic function defined in (\ref{def of inttrans}). Then for all $\trvar \in \Zp$,
\begin{align}
\inttrans{\psi}(\trvar) &= \int_{\Zp} (\one - \beta_1)^{\star\trvar}\,d\mu(H_\psi) \label{inttrans as integral: first equality} \\
&= \int_{\Zp} \Big((\one - \beta_1)^{\star\trvar} \star \psi\Big)\,d\mu(H_\one) . \label{inttrans as integral: second equality}
\end{align}
\end{theorem}

\begin{proof}
Consider the function
\defmap{\mathrm{Int}(\psi)}{\Zp}{\Cp}{\trvar}{\int_{\Zp} (\one - \beta_1)^{\star\trvar}\,d\mu(H_\psi)}[.]
We first show that $\mathrm{Int}(\psi)$ is continuous. Write $H_\psi(t) = \sum_{n = 0}^\infty b_n t^n$, and let $N = \sup_{n \geq 0} |b_n|$. Note that the supremum is finite because $\psi$ is bounded. If $\eps > 0$, then by Proposition~\ref{T is uniformly continuous}, there is $\delta > 0$ such that $\|\spwr(\trvar_1) - \spwr(\trvar_2)\|_M < \eps/N$ whenever $|\trvar_1 - \trvar_2| < \delta$. Take such $\trvar_1,\trvar_2$, and write $\spwr(\trvar_1) - \spwr(\trvar_2) = \sum_{n = 0}^\infty a_n\beta_n$, where $a_n \to 0$. Then
\begin{align*}
\left|\int_{\Zp} \spwr(\trvar_1)\,d\mu(H_\psi) - \int_{\Zp} \spwr(\trvar_2)\,d\mu(H_\psi)\right| &= \left|\int_{\Zp} (\spwr(\trvar_1) - \spwr(\trvar_2))\,d\mu(H_\psi)\right| \\
&= \left|\sum_{n = 0}^\infty a_n b_n\right| \\
&\leq \left(\max_{n \geq 0} |a_n|\right)\left(\sup_{n \geq 0} |b_n|\right) \\
&= \|\spwr(\trvar_1) - \spwr(\trvar_2)\|_M \cdot N \\
&< \eps .
\end{align*}
This proves continuity.

Now, for all $m \in \Z$,
\begin{align*}
\inttrans{\psi}(m) &= \calS^m(\psi)(-1) \eqcom{by Theorem~\ref{power of S and inttrans}} \\
&= ((\one - \beta_1)^{\star m} \star \psi)(-1) \eqcom{by Corollary~\ref{calS corr to 1 - beta_1}} \\
&= \pair{(\one - \beta_1)^{\star m}}{\psi} \eqcom{by Proposition~\ref{star and pairing}} \\
&= \int_{\Zp} (\one - \beta_1)^{\star m}\,d\mu(H_\psi) \eqcom{by definition} \\
&= \mathrm{Int}(\psi)(m) ,
\end{align*}
so because both $\inttrans{\psi}$ and $\mathrm{Int}(\psi)$ are continuous, they are equal as functions on $\Zp$. We thus have (\ref{inttrans as integral: first equality}), and then (\ref{inttrans as integral: second equality}) follows by Proposition~\ref{pairing as integral of product}.
\end{proof}

\begin{remark}
We observe a similarity between the integral transform
\[ \inttrans{\psi}(s) = \int_{\Zp} \Big((\one - \beta_1)^{\star\trvar} \star \psi\Big)\,d\mu(H_\one) \]
and the Laplace transform
\[ \lt{f}(s) = \int_0^\infty e^{-st}f(t)\dx[t] = \int_0^\infty (e^{-t})^s f(t)\dx[t] .\]
The equation $\inttrans{\sumf\psi} = -\calS^{-1}(\inttrans{\psi})$ satisfied by the integral transform $\inttmap$ (see Proposition~\ref{sumf and inttmap}) is akin to the equation $\mathcal{L}\{tf(t)\} = -\mathcal{L}\{f\}'$ satisfied by the Laplace transform. In our situation, the derivative operator is replaced by $\calS^{-1}$, and the multiplication-by-$t$ map is replaced by the $\sumf$ operator. We will also see in Section~\ref{sec: inttmap and conv} an analogue of the Laplace-transform formula $\lt{f \ast g} = \lt{f}\lt{g}$.
\end{remark}

\subsubsection{$\Gamma_p(s,r)$ in terms of the integral transform $\inttmap$} \label{int transf for gamma}

We show that O'Desky and Richman's $p$-adic incomplete $\Gamma$-functions can be expressed in terms of the integral transform $\inttmap$. We recall the functions $\gfn{r} = \calS^{-1}(g_r)$ and $\overline{\gamma}_r$ defined in Section~\ref{sec: p-adic inc}, where $r \in \lball{1}{1}[\Cp]$.

\begin{prop} \label{gfn in terms of an inttrans}
Let $r \in \lball{1}{1}[\Cp]$. Then $\gfn{r} = \shiftu(g_r)\inttrans{g_r}$.
\end{prop}

\begin{proof}
If $\trvar \in \Zp$,
\begin{align*}
\gfn{r}(\trvar) &= \sum_{n = 0}^\infty (\trvar)_n r^{\trvar - n} \eqcom{by Proposition~\ref{alternative for calS^{-1}}} \\
&= r^{\trvar + 1}\sum_{n = 0}^\infty (\trvar)_n r^{-1 - n} \\
&= \shiftu(g_r)(\trvar)\inttrans{g_r}(\trvar) \eqcom{by definition of $\inttmap$.}
\end{align*}
\end{proof}

\begin{prop}
Let $r \in \lball{1}{1}[\Cp]$, and define $\Gamma_p(-,r) = \overline{\gamma}_r$, where $\overline{\gamma}_r$ is as in Section~\ref{sec: p-adic inc}. Then for all $\trvar \in \Zp$,
\[ \Gamma_p(s,r) = \exp(r\tilde{p})r^\trvar\int_{\Zp} \Big((\one - \beta_1)^{\star(\trvar - 1)} \star g_r\Big)\,d\mu(H_\one) .\]
\end{prop}

\begin{proof}
If $\trvar \in \Zp$, then
\begin{align*}
\overline{\gamma}_r(s) &= \exp(r\tilde{p})\gfn{r}(s - 1) \\
&= \exp(r\tilde{p})r^s\inttrans{g_r}(s - 1) \eqcom{by Proposition~\ref{gfn in terms of an inttrans}} \\
&= \exp(r\tilde{p})r^\trvar\int_{\Zp} \Big((\one - \beta_1)^{\star(\trvar - 1)} \star g_r\Big)\,d\mu(H_\one) \eqcom{by Theorem~\ref{inttrans as integral}.}
\end{align*}
\end{proof}

We compare the above formula with the definition of the complex incomplete $\Gamma$-function:
\[ \Gamma(s,r) = \int_r^\infty x^{s - 1}e^{-x}\dx .\]

\subsection{Image of $\inttmap$} \label{sec: image of inttmap}

We determine which functions are in the image $\im(\inttmap)$ of the injective map $\inttmap : \cfns{\Zp}{\Cp} \ra \lan{\Zp}{\Cp}$. The following proposition will help in the task.

\begin{prop} \label{inttrans of limit}
If $(\psi_k)_{k \geq 0}$ is a sequence of functions in $\cfns{\Zp}{\Cp}$ converging uniformly to some $\psi$, then the sequence of functions $\inttrans{\psi_k}$ converges uniformly to $\inttrans{\psi}$.
\end{prop}

\begin{proof}
We let $\|\cdot\|$ denote the sup-norm for functions $\Zp \ra \Cp$. Let $\eps > 0$, and choose $N \geq 0$ such that $\|\psi_k - \psi\| < \eps$ for all $k \geq N$. Then for the same integers $k$,
\begin{align*}
\|\inttrans{\psi_k} - \inttrans{\psi}\| &= \left\|\sum_{n = 0}^\infty n!\,\psi_k(-1 - n)\beta_n - \sum_{n = 0}^\infty n!\,\psi(-1 - n)\beta_n\right\| \\
&= \left\|\sum_{n = 0}^\infty n!\,(\psi_k - \psi)(-1 - n)\beta_n\right\| \\
&\leq \|\psi_k - \psi\| \eqcom{because $\|\beta_n\| = 1$} \\
&< \eps .
\end{align*}
\end{proof}

Recall that $\beta_1(x) = x$, and that $\shiftu(\psi)(x) = \psi(x + 1)$.

\begin{lemma} \label{inttrans and nabla}
If $\psi \in \cfns{\Zp}{\Cp}$, then $\inttrans{\beta_1\shiftu^{-1}(\psi)} = -\nabla(\inttrans{\psi})$.
\end{lemma}

\begin{proof}
We begin with the definition of $\calS(\psi)$, namely, $\calS(\psi)(x) = \psi(x) - x\psi(x - 1)$, \ie $\calS(\psi) = \psi - \beta_1\shiftu^{-1}(\psi)$. Then
\begin{align*}
\inttrans{\beta_1\shiftu^{-1}(\psi)} &= \inttrans{\psi - \calS(\psi)} \\
&= \inttrans{\psi} - \inttrans{\calS(\psi)} \\
&= \inttrans{\psi} - \shiftu(\inttrans{\psi}) \eqcom{by Corollary~\ref{inttrans: calS and shiftu}} \\
&= -\nabla(\inttrans{\psi}) .
\end{align*}
\end{proof}

Before giving the next proposition, we observe that the functions $\frac{1}{n!}\nabla^n\qf$ map $\Zp$ into itself. Indeed, $\qf = \sum_{k = 0}^\infty k!\,\beta_k$, so
\[ \frac{1}{n!}\nabla^n\qf = \frac{1}{n!}\sum_{k = 0}^\infty (k + n)!\,\beta_n = \sum_{k = 0}^\infty \frac{(k + n)!}{n!}\,\beta_k ,\]
and $(n + k)!/n! \in \Z$.

\begin{prop} \label{inttrans in terms of mahler coeffs}
If $\psi = \sum_{n = 0}^\infty a_n\beta_n \in \cfns{\Zp}{\Cp}$, then
\[ \inttrans{\psi} = \sum_{n = 0}^\infty a_n\frac{(-1)^n}{n!}\nabla^n\qf .\]
\end{prop}

\begin{proof}
The binomial identity $x\ch{x - 1}{n} = (n + 1)\ch{x}{n + 1}$ says that $\beta_1\shiftu^{-1}(\beta_n) = (n + 1)\beta_{n + 1}$, so
\[ \inttrans{\beta_{n + 1}} = -\frac{1}{n + 1}\nabla(\inttrans{\beta_n}) \]
by Lemma~\ref{inttrans and nabla}. Therefore, because $\inttrans{\beta_0} = \qf$ by definition (see (\ref{def of inttrans})), induction on $n$ shows that
\[ \inttrans{\beta_n} = \frac{(-1)^n}{n!}\nabla^n\qf .\]
Hence, Proposition~\ref{inttrans of limit} shows that if $a_n \to 0$, then
\[ \inttrans*{\sum_{n = 0}^\infty a_n\beta_n} = \sum_{n = 0}^\infty a_n\frac{(-1)^n}{n!}\nabla^n\qf .\]
\end{proof}

\begin{cor} \label{desc of image of inttmap: nabla}
The image of the transform map $\inttmap : \cfns{\Zp}{\Cp} \ra \lan{\Zp}{\Cp}$ consists of the functions
\beq \label{general function in image of inttmap}
\Psi = \sum_{n = 0}^\infty a_n\frac{(-1)^n}{n!}\nabla^n\qf
\eeq
where $(a_n)_{n \geq 0}$ is a sequence in $\Cp$ converging to zero. A function $\Psi$ in the image determines the coefficients $a_n$ in (\ref{general function in image of inttmap}) uniquely.
\end{cor}

\begin{proof}
The first assertion is clear by Proposition~\ref{inttrans in terms of mahler coeffs}. As for the uniqueness of the coefficients for a given $\Psi$, if $\sum_{n = 0}^\infty a_n\frac{(-1)^n}{n!}\nabla^n\qf = 0$ where $a_n \to 0$, let $\psi = \sum_{n = 0}^\infty a_n\beta_n \in \cfns{\Zp}{\Cp}$. Then
\[ \inttrans{\psi} = \sum_{n = 0}^\infty a_n\frac{(-1)^n}{n!}\nabla^n\qf = 0 ,\]
so $\psi = 0$ by the injectivity of $\inttmap$, and therefore the $a_n$ are all zero.
\end{proof}

\begin{cor} \label{desc of image of inttmap: calS^{-n}}
For all $n \geq 0$, $\tfrac{1}{n!}\nabla^n\qf = \calS^{-n}(\qf)$. Therefore, $\im(\inttmap)$ consists of the functions
\[ \Psi = \sum_{n = 0}^\infty a_n(-1)^n\calS^{-n}(\qf) \]
where $(a_n)_{n \geq 0}$ is a sequence in $\Cp$ converging to zero. Again, the $a_n$ are determined uniquely by $\Psi$.
\end{cor}

\begin{proof}
We first recall Proposition~\ref{sumf and inttmap}, which, once iterated, says that $\inttrans{\sumf^n\psi} = (-1)^n\calS^{-n}(\inttrans{\psi})$. We apply this in the case where $\psi = \beta_0$. Then because $\beta_n = \sumf^n\beta_0$, we have
\[ \inttrans{\beta_n} = \inttrans{\sumf^n\beta_0} = (-1)^n\calS^{-n}(\inttrans{\beta_0}) = (-1)^n\calS^{-n}(\qf) ,\]
the last equality by Proposition~\ref{inttrans in terms of mahler coeffs}. On the other hand, $\inttrans{\beta_n} = (-1)^n\tfrac{1}{n!}\nabla^n\qf$ by Proposition~\ref{inttrans in terms of mahler coeffs} again. Therefore, $\calS^{-n}(\qf) = \tfrac{1}{n!}\nabla^n\qf$, so the rest of the statement to be proven follows from Corollary~\ref{desc of image of inttmap: nabla}.
\end{proof}

We note in passing that
\[ \calS^{-1}(\tfrac{1}{n!}\nabla^n\qf) = \calS^{-1}(\calS^{-n}(q)) = \calS^{-(n + 1)}(\qf) = \tfrac{1}{(n + 1)!}\nabla^{n + 1}(\qf) .\]

We make another observation regarding the image of $\inttmap$.

\begin{prop}
If $\lconst{\Zp}{\Cp} \subset \lan{\Zp}{\Cp}$ denotes the subspace of locally constant functions, then $\im(\inttmap) \cap \lconst{\Zp}{\Cp} = 0$.
\end{prop}

\begin{proof}
Let $\Psi = \inttrans{\psi}$ be in the intersection. Being locally constant, $\Psi$ is $\Z$-periodic of period $m = p^j$ for some $j \geq 0$ (see \cite[Chap.~4, Sect.~3.2]{robert:p-adic}, for example). That is, $\Psi(s + m) = \Psi(s)$ for all $s \in \Zp$. Then
\begin{align*}
\inttrans{\calS^m(\psi)}(s) &= \inttrans{\psi}(s + m) \eqcom{by Corollary~\ref{inttrans: calS and shiftu}} \\
&= \inttrans{\psi}(s) ,
\end{align*}
so $\inttrans{\calS^m(\psi)} = \inttrans{\psi}$. Hence, $\calS^m(\psi) = \psi$ by the injectivity of $\inttmap$, so $\psi = 0$ by Corollary~\ref{calS has infinite order}. (Note that $m = p^j > 0$.) Thus, $\Psi = 0$.
\end{proof}

\subsection{$\inttmap$ and convolution} \label{sec: inttmap and conv}

Recall the convolution product $\star$ on $\cfns{\Zp}{\Cp}$. We define a similar product on $\im(\inttmap)$, again via the algebra $\decayegf$ (see the beginning of Section~\ref{sec: conv of mahl}). Consider the mutually inverse maps
\beal
\im(\inttmap) &\ra& \decayegf \nonumber \\*
\sum_{n = 0}^\infty b_n\calS^{-n}(\qf) &\mapsto& \sum_{n = 0}^\infty b_n\,\frac{\gfvar^n}{n!} \eqcomb{where $b_n \ra 0$ in $\Cp$}[,] \\[2ex]
\decayegf &\ra& \im(\inttmap) \nonumber \\*
\sum_{n = 0}^\infty b_n\,\frac{\gfvar^n}{n!} &\mapsto& \sum_{n = 0}^\infty b_n\calS^{-n}(\qf) .
\eeal
These maps are well defined by Corollary~\ref{desc of image of inttmap: calS^{-n}}. Denote by $\diamond$ the product on $\im(\inttmap)$ inherited from $\decayegf$ via these maps. Concretely,
\[ \left(\sum_{k = 0}^\infty b_k\calS^{-k}(\qf)\right) \diamond \left (\sum_{l = 0}^\infty c_l\calS^{-l}(\qf)\right) = \sum_{n = 0}^\infty \left(\sum_{k = 0}^n \ch{n}{k}b_k c_{n - k}\right)\calS^{-n}(\qf) .\]
Then $(\im(\inttmap),+,\diamond)$ is a commutative $\Cp$-algebra with no zero divisors.

\begin{prop} \label{inttmap conv}
If $\phi,\psi \in \cfns{\Zp}{\Cp}$, then $\inttrans{\phi \star \psi} = \inttrans{\phi} \diamond \inttrans{\psi}$.
\end{prop}

\begin{proof}
Let the Mahler coefficients of $\phi$ and $\psi$ be $(b_k)_{k \geq 0}$ and $(c_l)_{l \geq 0}$ respectively. Then
\begin{align*}
\inttrans{\phi \star \psi} &= \inttrans*{\sum_{n = 0}^\infty \left(\sum_{k = 0}^n \ch{n}{k}b_k c_{n - k}\right)\beta_n} \\
&= \sum_{n = 0}^\infty (-1)^n\left(\sum_{k = 0}^n \ch{n}{k}b_k c_{n - k}\right)\calS^{-n}(\qf) \\
&= \sum_{n = 0}^\infty \left(\sum_{k = 0}^n \ch{n}{k}(-1)^k b_k(-1)^{n - k} c_{n - k}\right)\calS^{-n}(\qf) \\
&= \left(\sum_{k = 0}^\infty (-1)^k b_k \calS^{-k}(\qf)\right) \diamond \left(\sum_{l = 0}^\infty (-1)^l c_l \calS^{-l}(\qf)\right) \\
&= \inttrans{\phi} \diamond \inttrans{\psi} .
\end{align*}
\end{proof}

Proposition~\ref{inttmap conv} is akin to the convolution formula for the Laplace transform.

\section{Connection to the equation $F' + F = G$} \label{DE section}

We describe the role of the automorphism $\calS$ in the solution to the well-known differential equation $F' + F = G$, where $G \in \Cp\pwr{\gfvar}$ is a given power series with bounded coefficients. Of course, our intention is not to attempt to contribute anything to the general theory of $p$-adic differential equations, which has a vast literature extending far beyond this equation, but simply to highlight a connection.

It is worth remarking that, trivial though the equation $F' + F = G$ appears to be, it is not at first sight obvious that if $G \in \Zp\pwr{t}$, say, then the equation has a solution $F$ that is also in $\Zp\pwr{t}$ (which it does).

We have already seen that if $G(t) = \sum_{n = 0}^\infty \psi(n)\gfvar^n \in \Cp\pwr{\gfvar}$ has coefficients given by a continuous function $\psi : \Zp \ra \Cp$, then $F' + F = G$ where
\[ F(\gfvar) = \sum_{n = 0}^\infty \calS^{-1}(\widehat{\psi})(-1 - n)\gfvar^n ,\]
$\widehat{\psi}$ being the function $x \mapsto \psi(-1 - x)$. Just replace $\rfn$ by $\widehat{\psi}$ in Proposition~\ref{calS and gen funcs}. A similar fact holds even when the coefficients of $G$ are not assumed to vary continuously with $n$ and are instead assumed only to be bounded. We turn now to this generalization and show how the automorphism $\calS$ still appears.

Let $K$ be a complete subfield of $\Cp$. For a positive real number $M$, we define $\bfact{\Zp}{K}[M]$ to consist of the functions $\psi : \Zp \ra K$ having Mahler series of the form
\[ \psi = \sum_{n = 0}^\infty n!\,a_n\beta_n \eqcom{where $|a_n| \leq M$ for all $n \geq 0$.} \]
We also let
\[ \bfact{\Zp}{K} = \bigcup_{M > 0} \bfact{\Zp}{K}[M] .\]

\begin{prop} \label{calS and shiftu on bfact}
For each $M > 0$, the $\Cp$-linear automorphisms $\calS$ and $\shiftu$ of $\cfns{\Zp}{\Cp}$ restrict to $K$-linear automorphisms of $\bfact{\Zp}{K}[M]$.
\end{prop}

\begin{proof}
Let $\psi = \sum_{n = 0}^\infty n!\,a_n\beta_n$. Let us handle $\calS$ first. By Proposition~\ref{calS of mahler},
\begin{align}
\calS\left(\sum_{n = 0}^\infty n!\,a_n\beta_n\right) &= a_0\beta_0 + \sum_{n = 1}^\infty (n!\,a_n - n(n - 1)!\,a_{n - 1})\beta_n \nonumber \\
&= a_0\beta_0 + \sum_{n = 1}^\infty n!\,(a_n - a_{n - 1})\beta_n \nonumber \\
&= \sum_{n = 0}^\infty n!\,(a_n - a_{n - 1})\beta_n \label{calS of fact mahl}
\end{align}
if we let $a_{-1} = 0$. It is now clear that if $\psi \in \bfact{\Zp}{K}[M]$, then so is $\calS(\psi)$. Similarly, by Theorem~\ref{calS theorem},
\begin{align*}
\calS^{-1}\left(\sum_{n = 0}^\infty n!\,a_n\beta_n\right) &= \sum_{n = 0}^\infty \left(\sum_{k = 0}^n \frac{n!}{k!}\,k!\,a_k\right)\beta_n \\
&= \sum_{n = 0}^\infty n!\left(\sum_{k = 0}^n a_k\right)\beta_n ,
\end{align*}
so again, $\calS^{-1}(\psi) \in \bfact{\Zp}{K}[M]$ if $\psi$ is.

As for $\shiftu$, the easily verified equalities
\begin{align}
\shiftu\left(\sum_{n = 0}^\infty n!\,a_n\beta_n\right) &= \sum_{n = 0}^\infty n!\,\big((n + 1)a_{n + 1} + a_n\big)\beta_n , \label{shiftu of fact mahl} \\
\shiftu^{-1}\left(\sum_{n = 0}^\infty n!\,a_n\beta_n\right) &= \sum_{n = 0}^\infty n!\left(\sum_{k = n}^\infty (-1)^{k - n}\frac{k!}{n!}\,a_k\right)\beta_n \nonumber
\end{align}
show that if $\psi \in \bfact{\Zp}{K}$, then so are $\shiftu(\psi)$ and $\shiftu^{-1}(\psi)$.
\end{proof}

Now define
\defmap{\qdo}{\Cp\pwr{\gfvar}}{\Cp\pwr{\gfvar}}{F}{(1 - \gfvar)(F' + F)}[.]
Also, if $\lfn \in \cfns{\Zp}{\Cp}$ with Mahler series $\lfn = \sum_{n = 0}^\infty n!\,a_n\beta_n$, define
\[ J_\lfn(\gfvar) = \sum_{n = 0}^\infty (a_n - a_{n - 1})\gfvar^n ,\]
where, for convenience, we have set $a_{-1} = 0$.

\begin{prop} \label{J and calS}
If $\lfn \in \cfns{\Zp}{\Cp}$, then $\qdo(J_\lfn) = J_{\shiftu \of \calS(\lfn)}$.
\end{prop}

\begin{proof}
Let $\lfn = \sum_{n = 0}^\infty n!\,a_n\beta_n$. One sees immediately from (\ref{calS of fact mahl}) and (\ref{shiftu of fact mahl}) that $\shiftu \of \calS(\lfn) = \sum_{n = 0}^\infty n!\,b_n\beta_n$ where
\[ b_n = (n + 1)a_{n + 1} - na_n - a_{n - 1} .\]
Note that, if we define $a_{-2} = 0$, then $b_{-1} = 0$. (We have already taken $a_{-1}$ to be $0$.)

Now,
\begin{align*}
J_\lfn'(\gfvar) + J_\lfn(\gfvar) &= \sum_{n = 0}^\infty (n + 1)(a_{n + 1} - a_n)\gfvar^n + \sum_{n = 0}^\infty (a_n - a_{n - 1})\gfvar^n \\
&= \sum_{n = 0}^\infty ((n + 1)a_{n + 1} - na_n - a_{n - 1})\gfvar^n \\
&= \sum_{n = 0}^\infty b_n\gfvar^n \eqcom{by the above} \\
&= \sum_{n = 0}^\infty \left(\sum_{k = 0}^n (b_k - b_{k - 1})\right)\gfvar^n \eqcomb{$b_{-1} = 0$} \\
&= \frac{1}{1 - \gfvar}\sum_{n = 0}^\infty (b_n - b_{n - 1})\gfvar^n \\
&= \frac{1}{1 - \gfvar}J_{\shiftu \of \calS(\lfn)}(\gfvar) .
\end{align*}
\end{proof}

For a real number $M > 0$, let $K\pwr{\gfvar}_{\leq M}$ denote the set of power series $G(t) = \sum_{n = 0}^\infty b_n\gfvar^n$ with $|b_n| \leq M$ for all $n$, and let $K\pwr{\gfvar}_b = \bigcup_{M > 0} K\pwr{\gfvar}_{\leq M}$, which consists of the power series with bounded coefficients.

Before stating the next proposition, we point out \cite[Theorem~2.1]{gefter-goncharuk:hurwitz}, which considers a related differential equation in the situation where $A = \Zp$, but without using the automorphism $\calS$.

\begin{prop} \label{DE prop}
Let $X$ be equal either to $K\pwr{\gfvar}_{\leq M}$ with $M > 0$, or to $K\pwr{\gfvar}_b$, and let $A = \{x \in K \sat |x| \leq 1\}$. Then
\bea
X &\ra& X \\*
F &\mapsto& F' + F
\eea
is an $A$-module automorphism of $X$. In particular, $F \mapsto F' + F$ is an automorphism of the $A$-module $A\pwr{\gfvar}$. If $G(\gfvar) = \sum_{n = 0}^\infty b_n\gfvar^n \in X$ and $\rfn = \sum_{n = 0}^\infty n!\,b_n\beta_n$, then the unique $F \in X$ such that $F' + F = G$ is $J_\lfn$ where $\lfn = \calS^{-1} \of \shiftu^{-1}(\rfn)$.
\end{prop}

\begin{proof}
Because $K\pwr{\gfvar}_b = \bigcup_{M > 0} K\pwr{\gfvar}_{\leq M}$ and $A\pwr{\gfvar} = K\pwr{\gfvar}_{\leq 1}$, it is enough to prove everything in the case where $X = K\pwr{\gfvar}_{\leq M}$.

For injectivity, suppose that $F \in K\pwr{\gfvar}_{\leq M}$ satisfies $F' + F = 0$. Then $F(\gfvar) = c\exp(-\gfvar)$ for some $c \in K$, but $\exp(-\gfvar)$ does not have bounded coefficients, so $c = 0$.

For surjectivity, let $G \in K\pwr{\gfvar}_{\leq M}$, say $G(\gfvar) = \sum_{n = 0}^\infty b_n\gfvar^n$. Then
\[ (1 - \gfvar)G(\gfvar) = \sum_{n = 0}^\infty (b_n - b_{n - 1})\gfvar^n = J_\rfn(\gfvar) \]
where $\rfn = \sum_{n = 0}^\infty n!\,b_n\beta_n$ (and $b_{-1} = 0$). Note that $\rfn \in \bfact{\Zp}{K}[M]$, because $|b_n| \leq M$ for all $n$. Now let $\lfn = \calS^{-1} \of \shiftu^{-1}(\rfn)$, which is in $\bfact{\Zp}{K}[M]$ by Proposition~\ref{calS and shiftu on bfact}. Then $J_\lfn \in K\pwr{\gfvar}_{\leq M}$, and Proposition~\ref{J and calS} says that
\[ \qdo(J_\lfn) = J_{\shiftu \of \calS(\lfn)} = J_\rfn .\]
Hence, remembering that $J_\rfn(\gfvar) = (1 - \gfvar)G(\gfvar)$, we have $J_\lfn'(\gfvar) + J_\lfn(\gfvar) = G(\gfvar)$.
\end{proof}

\bigskip

\noindent\textbf{Acknowledgements}

\medskip

\noindent The author would like to thank Al Weiss and Neil Dummigan for helpful conversations.

\end{document}